\documentclass[a4paper,reqno]{amsart}
\usepackage[utf8]{inputenc}

\usepackage[left=1in,right=1in,top=1in,bottom=1in]{geometry}
\usepackage{amsfonts}
\usepackage{amsmath}
\usepackage{amssymb}
\usepackage{multirow}
\usepackage{graphics}                 
\usepackage{xcolor}                    
\usepackage{xspace}
\usepackage{hyperref}                 
\usepackage{cancel}

\usepackage{tikz}

\newcommand{\ebf}{{\bf e}}

\newcommand{\ubf}{{\bf u}}
\newcommand{\vbf}{{\bf v}}
\newcommand{\xbf}{{\bf x}}
\newcommand{\ybf}{{\bf y}}
\newcommand{\zbf}{{\bf z}}

\def\Xbf{{\mathbf X}}

\def\wf{\widetilde{f}}

\def\wu{\widetilde{u}}

\def\wz{\widetilde{z}}

\def\wA{\widetilde{A}}

\def\wF{\widetilde{F}}

\def\wX{\widetilde{X}}
\def\wY{\widetilde{Y}}

\newcommand{\ybs}{{\boldsymbol y}}

\def\Cbb{{\mathbb C}}
\def\Ebb{{\mathbb E}}

\def\Nbb{{\mathbb N}}
\def\Pbb{{\mathbb P}}
\def\Rbb{{\mathbb R}}

\def\Bcal{{\mathcal B}}

\def\Fcal{{\mathcal F}}

\def\Ical{{\mathcal I}}
\def\Jcal{{\mathcal J}}
\def\Lcal{{\mathcal L}}

\def\Ucal{{\mathcal U}}

\def\Xcal{{\mathcal X}}
\def\Ycal{{\mathcal Y}}

\newcommand{\dif}[1]{\mathrm{d} #1}

\DeclareMathOperator*{\argmin}{argmin}

\DeclareMathOperator*{\supp}{supp}
\DeclareMathOperator*{\bsupp}{B-supp}

\newcommand{\vertiii}[1]{{\left\vert\kern-0.25ex\left\vert\kern-0.25ex\left\vert #1 
    \right\vert\kern-0.25ex\right\vert\kern-0.25ex\right\vert}} 

\newcommand{\Tscheb}{Chebyshev\xspace}

\newcommand{\jlb}[1]{{\color{black}{#1}}}

\usepackage{todonotes} 
\makeatletter
\providecommand\@dotsep{5}
\def\listtodoname{List of Todos}
\def\listoftodos{\@starttoc{tdo}\listtodoname}
\makeatother

\usepackage{amsthm}
\newtheorem{acorol}{Corollary}
\newtheorem{adef}{Definition}
\newtheorem{atheorem}{Theorem}
\newtheorem{aprop}{Proposition}
\newtheorem{alemma}{Lemma}
\theoremstyle{remark}
\newtheorem{assumption}{Assumption}

\newtheorem{rmk}{Remark}

\numberwithin{equation}{section}
\numberwithin{atheorem}{section}
\numberwithin{acorol}{section}
\numberwithin{rmk}{section}
\numberwithin{aprop}{section}

\title{Weighted block compressed sensing for parametrized function approximation}
\thanks{
This work started while the author was with the Chair C for Mathematics (Analysis), RWTH Aachen University, supported in part by the ERC grant StG 258926. 
The author would like to thank the Hausdorff Research Institute for Mathematics and the support of the Clay Mathematics Institute for his visit to the CRM to attend the IRP Constructive Approximation and Harmonic Analysis where part of this work has been done.
The author personally thanks Holger Rauhut for suggestions.
}

\author{Jean-Luc Bouchot}
\address{School of Mathematics and Statistics, Beijing Institute of Technology} 
\email{jlbouchot@bit.edu.cn} 
\date{\today}

\begin{document}

\begin{abstract}
In this paper we extend results taken from compressed sensing to recover Hilbert-space valued vectors. 
This is an important problem in parametric function approximation in particular when the number of parameters is high. 
By expanding our target functions in a polynomial chaos and assuming some compressibility of such an expansion, we can exploit structured sparsity (typically a group sparsity structure) to recover the sequence of coefficients with high accuracy. 

While traditional compressed sensing would typically expect a number of snapshots scaling exponentially with the number of parameters, we can beat this dependence by adding weights. 
This anisotropic handling of the parameter space permits to compute approximations with a number of samples scaling only linearly (up to log-factors) with the intrinsic complexity of the polynomial expansion. 

Our results are applied to problems in high-dimensional parametric elliptic PDEs. 
We show that under some weighted uniform ellipticity assumptions of a parametric operator, we are capable of numerically approximating the full solution (in contrast to the usual quantity of interest) to within a specified accuracy. 
\end{abstract}

\maketitle

\section{Introduction}
The problem of approximating parametric functions is analyzed through the lens of compressed sensing and model based sparse recovery~\cite{baraniuk2010model}. 
Parametric function approximation may be seen from \jlb{various perspectives such as vector valued-function approximation or parametrized function} .
The mathematical formulation reads as follows: 
given a (time-)space domain $\Omega \subset \Rbb^n$ and a parameter space $\Ucal$, find an approximation $\wu$ for $u : \Omega \times \Ucal \to \Rbb$ uniformly for all parameter $\ybs \in \Ucal$. 
A case that may be of interest is the case of time-dependent dynamical systems, where the time coordinate can be seen as a parameter. 
Even though what is presented here should have its own interest, this research is motivated by problems arising in parametric PDEs; an application we detail in great depth in Section~\ref{sec:PDEsExamples}. 
We recommend reading~\cite{Chkifa14Breakingcurse,Cohen15HighDPDEs} for some background on the challenges of high-dimensional parametric PDEs. 
The goal of this research is to recover the unknown mapping $u$ from the knowledge of its values $u(\ybs^{(i)})$ at certain (as few as possible) sampling points $\ybs^{(i)}$. 
By considering a polynomial expansion of the solution in the parameter space, we can model this problem as a linear inverse problem
\emph{Recover $u$ such that $Au = b$} where $A = (T_\nu(\ybs^{(i)}))_{\nu \in \Lambda; 1 \leq i \leq m}$ for a certain orthonormal system $T_\nu$ and where $b$ is a (vector valued-)\emph{vector} containing approximations of the solutions. 

Obviously without further information, the problem of recovering a vector-valued function in continuous space is not possible. 
To facilitate, we develop a framework based on compressed sensing and (structured) sparse approximation to solve it. 
We see that this approach allows, under some rather general assumptions, to recover parametric functions, even in high parameter spaces, thereby breaking the curse of dimensionality.

To be more precise, we combine ingredients from compressed sensing with the added flavor of structured sparsity. 
This structure will later take two precise forms: 1) a group sparsity model, in which one specifies globally that the spatial coordinates are activated \jlb{for the same pattern of multi-indices of parameters} and 2) a weighted sparse model, in which we add some anisotropy to the parameters. 
Basically, by favoring parameters which we know are of importance (via the weights), we then apply some structured recovery (in form of groups) such that we can recover the vector valued function. 

Ideas of weighted compressed sensing have already been developed for the recovery of scalar-valued functions, see~\cite{Rauhut13wCS}. 
In the case of interpolation problems, it allows for a mix between an efficient sparse representation and the recovery of a smooth function. 
Among other things, this shows a certain robustness to strong oscillations at the boundaries for interpolation. 

The block structure, on the other hand, has been encountered under the umbrella of model-based compressed sensing and has been known to be important for instance in the context of multi-band signal recovery. 
In such applications, one wants to recover functions which are sparse in a Fourier spectrum, albeit with its active frequencies grouped together in small sub-bands. 
A great advantage of group-sparse and block-sparse recovery methods is that it allows for sampling rates which are far below the traditional Nyquist sampling rate. 
However, no attempt to use structured compressed sensing with prior support information has been made so far.

The goal of this note is two-fold: first we try to bridge the gap between structured sparsity and sparse recovery with (partial, prior) support information. 
Second we apply our methods to \textit{uncertainty quantification} (UQ) in high-dimensional parametric (elliptic) PDEs. 
The latter has been a growing area of research in the last decade where methods based on 
adaptive stochastic Galerkin methods, as developed
in \cite{EGSZ14_1230,EGSZ14_1045,G13_531},
reduced basis approaches (see, eg., \cite{BCDDPW11,BufMadPat12}),
adaptive Smolyak discretizations \cite{SS13_813,SS14_1117}, 
adaptive interpolation methods \cite{chkifa2014polyDownward} 
as well as sampling methods \cite{TangIacc2014} have been analyzed.
UQ generally deals with computing moments of a \textit{quantity of interest} (QoI) of a certain physical phenomenon. 
QoI are often real-valued mapping computed from solutions of a certain PDE. 
One may consider for instance an average value of heat in a material. 
This is a challenging problem due to the (potential) high-dimensionality of the parameter space and requires numerical integration in very high-dimensions; considering the point wise approximation simply adds to the computational complexity of the problem. 
In this paper, we see that it is indeed possible to use our approximation techniques to compute pointwise estimations of solutions to families of operators indexed by countably many parameters. 
Under rather general assumptions (a generalization of a uniform ellipticity assumption), we can compute an approximation of solution uniformly for all parameters, by simply using joint sparsity of a generalized polynomial chaos expansion. 
Finally, computing moments of \textit{any} QoI can be done a posteriori without the need of numerically solving costly PDEs. 

A great advantage of the approach presented -- shall it be in the context of parametric PDEs or more generally for vector valued function approximation -- is that it can easily be computed in parallel. 
Up to a single, relatively computationally cheap, step, all computations are independent from one another. 
This is particularly interesting for problems in which the dimensionality of the spatial coordinate is higher than one. 

As suggested above, the paper combines group sparsity with prior support information in the form of weights. 
Signal recovery with partial support information has already been studied in the past, among others using certain hard-coded weights~\cite{jacques2010CSwSupp,khajehnejad2011wCS,friedlander2012recovering} and in form of a Bayesian problem~\cite{xu2010compressive}. 
Our work heavily relies on the work of Rauhut and Ward~\cite{Rauhut13wCS} in which weighted sparse approximation of real-valued functions is analyzed. 
In Section~\ref{sec:weightedblockCS} we introduce our general signal model and review some necessary basics about compressed sensing. 
The section ends with some examples in which such a structure may be useful. 
Section~\ref{sec:recovery} derives results for which recovery is possible, under certain conditions on the sensing matrix $A$. 
Section~\ref{sec:matrices} shows the existence of matrices which fulfill the conditions introduced in the previous section. 
We concentrate mainly on two types: 1) matrices in which all the entries are generated from independent (not necessarily identical) sub-Gaussian random variables, and 2) matrices obtained by tensor products of the identity and matrices obtained by sampling orthonormal systems. 
Our results for the former case generalize known results for the Gaussian, unweighted case and allows to add prior information to the non-linear recovery procedure. 
The analysis of the latter shows no strong advantages over known literature, but are important in order to derive the recovery of the full solution to a high-dimensional parametric PDEs, which we derive in Section~\ref{sec:PDEsExamples}.
This whole last section is dedicated to analyzing this problem and should be of independent interest in computational sciences. 
Our results, whose pinnacle is stated as Theorem~\ref{thm:EDPrecovery}, can be summarized as follows:
\begin{atheorem}
Assume $A(\ybs)$ is an affine family of operators parametrized by $\ybs$. 
Assume moreover that there exists a $p \in (0,1)$ and a sequence of weights $\vbf = (v_j)_{j \in \Nbb}$ such that $A$ is compressible and \textit{summable} (these two notions will be made precise in due time) in an $\ell^p_\vbf$ space. 
Then there exists a sequence of weights $\omega(\vbf)$ such that the solution to an $\ell_\omega^1$ minimization problem yields a solution $\wf$ satisfying the following approximation bounds 
$$
\|u(\ybs) - \wu(\ybs)\| \leq Cs^{1-1/p}\|u\|_{\ell^p_\omega} + D\varepsilon, \quad \text{uniformly for all } \ybs.
$$
\end{atheorem}
The meaning of the various norms will be clarified in the later sections. 
This theorem basically ensures us the existence of a tractable solution such that all parametric (in $\ybs$) solutions can be computed for all spatial coordinates (i.e. for all $x \in \Omega$) \jlb{with uniform (with respect to $\ybs$) error bounds}.

\section{Structured sparsity and weighted sensing}
\label{sec:weightedblockCS}
This section aims at introducing/recalling tools from compressed sensing and derive their use in the particular signal model of weighted group sparse signals.  

\subsection{Signal model}

\begin{adef} 
A vector $\xbf \in \Rbb^N$ is said to have a $\Bcal$ block structure if it is interpreted as a concatenation of $B$ vectors $\xbf = \left( \xbf[b] \right)_{b \in \Bcal}$, where $\xbf[b] \in \Rbb^{d_b}$ with $d_1 +d_2 + \cdots + d_B = N$. 
\end{adef}

As it is common, we can define block-based norms, for some $p,q > 0$,
\begin{equation*}
    \|\xbf\|_{q,p}^{(\Bcal)} := \left(\sum_{b \in \Bcal}\|\xbf[b]\|_q^p\right)^{1/p}.
\end{equation*}
The superscript $\Bcal$ is emphasized here to remember that the norm of a block signal (obviously) depends on the partition chosen. 
To avoid overly complicated notations, we drop this from now on, but the reader should keep this reliance in mind. 
In particular, as $p \searrow 0$ we obtain the block sparsity of a signal $\xbf$
\begin{equation*}
    \|\xbf\|_{q,0} = \left| \left\{ b \in \Bcal: \|\xbf[b]\|_q \neq 0 \right\} \right|.
\end{equation*}
This definition, while still dependent on the block structure $\Bcal$ chosen, is independent of the inner $q$-norm.

Note that one could choose to extend this definition and what follows to the case of infinite (but countable) dimensional vectors by simply indexing over a countable set $\Lambda$ instead of $\Bcal$. 
This is an interesting topic indeed from a theoretical perspective but cannot be directly implemented on a computer. 
We leave it to the interested reader to convince themselves that all results written in this paper extend easily and show in Section~\ref{sec:PDEsExamples} how one may truncate the infinite expansion, while keeping the accuracy in the recovery procedures in the context of numerical approximation of solutions of parametric PDEs.


As it was done for traditional compressed sensing, one can also develop a theory of weighted compressed sensing when dealing with the block sparse signal model. 
To this end and throughout this note, given a block structure $\Bcal$ containing $B$ blocks, we consider a sequence of weights $\omega := (\omega_1, \cdots, \omega_B)$\footnote{Again, one could write $\omega = (\omega_j)_{j \in \Lambda}$ and extend everything to a countable family of blocks.} of real numbers $\omega_i \geq 1$. 
The $p$-block-norms can be redefined in a weighted fashion
\begin{equation*}
    \|\xbf\|_{q,p}^ {(\omega)} := \left(\sum_{b \in \Bcal}\omega_b^{2-p}\| \xbf[b] \|_q^p\right)^{1/p}.
\end{equation*}
Note here that we simplify the exposition by writing $\omega_b$ for a $b$ either representing the index in the sequence of blocks, or the block directly.
And similarly, the weighted block sparsity is defined as 
\begin{equation*}
    \|\xbf\|_{q,0}^{(\omega)} = \sum_{b \in \Bcal: \xbf[b] \neq 0}\omega_b^2.
\end{equation*}
This weighted block sparsity being independent of the inner norm chosen, we will simply write $\|\xbf\|_0^{(\omega)}$ or $\omega(\xbf)$.
Before we dig deeper in the results, we have to introduce another notion, that of \emph{block-support}. 
It is defined as 
\begin{equation}
\label{def:bsupp}
\bsupp(\xbf) := \{ b \in \Bcal: \xbf[b] \neq 0 \}.
\end{equation}
Given a block-support set $S \subset \{1, \cdots, B \}$, its weighted \jlb{cardinality} is defined as 
\begin{equation*}
\omega(S) := \sum_{i \in S} \omega_i^2.
\end{equation*}
The notation $\xbf[S]$ will denote either the (block) vector equal to $\xbf$ on each block with indices in $S$ and $0$ elsewhere, or the extraction of the blocks indexed by $S$. 
The intentions are clear in the context. 
Throughout, we want to approximate signals in 
\begin{equation}
\label{eq:signalSpace}
S_{q,p}^{(\omega)} := \{ \xbf \in \Rbb^N: \|\xbf\|_{q,p}^{(\omega)} < \infty \}.
\end{equation}

As is usual in non-linear approximation, we assess the quality of our estimations via Lebesgue-type inequalities and hence extend the definition of best-$s$-term $\ell_\omega^p$ approximation to the weighted block case:
$$
\sigma_s(\xbf)_{q,p}^{(\omega)} = \inf_{\zbf: \omega(\zbf) \leq s}\|\xbf-\zbf\|_{q,p}^{(\omega)}.
$$
Finding the best weighted approximation is an NP-hard problem. 
An easier approximation to find is the so-called quasi-best weighted $s$ term approximation defined as follows. 
Let $\widetilde{\xbf}$ be the non-decreasing weighted block rearrangements of $\xbf$. I.e., there exists a permutation $\pi$ in $\{1, \cdots, B\}$ such that $\widetilde{\xbf}[i] = \xbf[\pi(i)]$ with $\|\xbf[\pi(i)]\|^p\omega(\pi(i))^{-p} \geq \|\xbf[\pi(j)]\|^p\omega(\pi(j))^{-p}$ for all $j \leq i$. 
For a (weighted) sparsity value of $s \geq \|\omega\|_\infty^2$, define $k_s := \max\{k: \sum_{i=1}^k \omega_{\pi(i)}^2 \leq s\}$ and then $S := \{\pi(1), \cdots, \pi(k_s)\}$.
The quasi-best weighted $s$ block approximation is then defined as $\widetilde{\xbf}[S]$. 
It follows that the best quasi-approximation fulfills
\begin{equation*}
\widetilde{\sigma}_s(\xbf)_{q,p}^{(\omega)} := \|\xbf - \widetilde{\xbf}[S]\|_{q,p}^{(\omega)} \geq \sigma_s(\xbf)_{q,p}^{(\omega)}.
\end{equation*}
We also have the following inequality: 
\begin{equation}
\label{eq:estimateQuasiweighted}
\widetilde{\sigma}_{3s}(\xbf)_{q,p}^{(\omega)} \leq \sigma(\xbf)_{q,p}^{(\omega)}. 
\end{equation}
Verifying inequality~\eqref{eq:estimateQuasiweighted} is done in a similar way as in~\cite{Rauhut13wCS} with the appropriate changes. 
We postpone the details to Appendix~\ref{app:quasiapprox} for the curious readers, as they merely are adaptions of known proofs. 

Stechkin's approximation bound can also be extended to the weighted block case. 
\begin{aprop}
\label{prop:Stechkin}
Let $\xbf$ be a $B$-block signal and let $\omega = (\omega_1, \cdots, \omega_B)$ be a sequence of weights with $\omega_i \geq 1$. 
Let $s > \|\omega\|_\infty^2$ denote a weighted sparsity. 
Then for $q < p \leq 2$ and $r$ denoting any norm on the blocks, the following Stechkin's bound holds
\begin{equation}
\label{eq:Stechkin}
\sigma_s(\xbf)_{r,p}^{(\omega)} \leq \widetilde{\sigma}_s(\xbf)_{r,p}^{(\omega)} \leq (s - \|\omega\|_\infty^2)^{\frac{1}{p} - \frac{1}{q}}\|\xbf\|_{r,q}^{(\omega)}.
\end{equation}
\end{aprop}
Once again, the proof is postponed to Appendix~\ref{app:quasiapprox}.

\subsection{Some particular examples of block structures}

The abstract structure described above gives rise to some particular examples. 

\begin{enumerate}
\item Traditional weighted compressed sensing: assuming all the sub-vectors to be composed of a single value gives back the weighted compressed sensing setting. 
Moreover, setting all weights to $1$ yields the original compressed sensing set-up. 
\item Similarly, keeping the block-structure and uniformly setting all weights equal to a given constant yields the by-now classical theory of structured sparsity in a union of subspaces. 
\item Joint sparsity, in which a set of unknown vectors of same size share a common sparsity pattern can also be modeled. 
Consider the matrix $\Xbf = [ \xbf[b]^T ]_{b \in \Bcal}$ created by stacking row-wise all the elements $\xbf[b]$. 
Then assuming a block sparse structure over $\xbf$ implies that the columns of $\Xbf$ share the same sparsity pattern. 
This is further developed in Section~\ref{sec:PDEsExamples} where we analyze the use of this weighted block sparse recovery for high-dimensional parametric PDEs.
\end{enumerate}

\section{Recovery guarantees}
\label{sec:recovery}
This section aims at extending and merging known results from block sparsity on the one hand and weighted sparsity on the other, to the framework of weighted block sparse signals. 
\begin{adef}
\label{def:wbrnsp}
Let $\Bcal$ be a block structure for $\Rbb^N$ with the associated weight sequence $\omega = (\omega_b)_{b \in \Bcal}$.
A matrix $A \in \Rbb^{m \times N}$ is said to fulfill the weighted block $\ell^p$ robust null space property ($\ell_\omega^p-$BRNSP) with respect to $\Bcal$ of order $s \geq \|\omega\|_\infty$ and constants $\rho \in (0,1)$ and $\tau > 0$ if 
\begin{equation*}
\| \xbf[S]\|_{2,p}^{(\omega)} \leq \frac{\rho}{s^{1-1/p}}\|\xbf[S^c]\|_{2,1}^{(\omega)} + \tau\|A\xbf\|_2
\end{equation*}
holds for all \jlb{$\xbf \in \Rbb^N$ and all} block support set $S$ such that $\omega(S) \leq s$.
\end{adef}

Not surprisingly, the following result shows that the $\ell^1_\omega$-BRNSP is necessary and sufficient for a robust recovery via weighted block $\ell_\omega^1-$minimization.
\begin{aprop}
\label{prop:lemmaL1NSP}
Let $B \in \Nbb$ be a number of blocks of $\Rbb^N$. 
Suppose that $A$ fulfills the $\ell_\omega^1$-BRNSP of order $s$ with constant $\tau > 0$ and $\rho \in (0,1)$. Then, for any $\zbf, \xbf \in \Rbb^{N}$, we have 
\begin{equation*}
\|\zbf-\xbf\|_{2,1}^{(\omega)} \leq \frac{1+\rho}{1-\rho}\left( \|\zbf\|_{2,1}^{(\omega)} - \|\xbf\|_{2,1}^{(\omega)} + 2\sigma_s(\xbf)_{2,1}^{(\omega)} \right) + \frac{2\tau}{1-\rho}\|A(\xbf-\zbf)\|_2.
\end{equation*}
\end{aprop}

\begin{proof}
Let $S$ be the block-support of the best weighted $s$ block approximation, i.e. $S$ is such that $\sigma_s(\xbf)_{2,1}^{(\omega)} = \|\xbf - \xbf[S]\|_{2,1}^{(\omega)} = \|\xbf[S^c]\|_{2,1}^{(\omega)}$.
$$
\|\xbf\|_{2,1}^{(\omega)} + \left\| \left( \xbf-\zbf \right)[S^c] \right\|_{2,1}^{(\omega)} \leq 2\|\xbf[S^c]\|_{2,1}^{(\omega)} + \|\xbf[S]\|_{2,1}^{(\omega)} + \|\zbf[S^c]\|_{2,1}^{(\omega)} \leq 2\sigma_s(\xbf)_{2,1}^{(\omega)} + \|\left(\xbf-\zbf\right)[S]\|_{2,1}^{(\omega)} + \|\zbf\|_{2,1}^{(\omega)}.
$$ 
With $\vbf := \xbf-\zbf$ it follows directly that 
\begin{equation}
\label{eq:estimateNSPmiddle}
\|\vbf[S^c]\|_{2,1}^{(\omega)} \leq \|\vbf[S]\|_{2,1}^{(\omega)} + \|\zbf\|_{2,1}^{(\omega)} - \|\xbf\|_{2,1}^{(\omega)} + 2 \sigma_s(\xbf)_{2,1}^{(\omega)}.
\end{equation}
Using Def.~\ref{def:wbrnsp} of the $\ell^p_\omega-$BRNSP for $p = 1$, it follows 
$$
\|\vbf[S]\|_{2,1}^{(\omega)} \leq \rho\|\vbf[S^c]\|_{2,1}^{(\omega)} + \tau\|A\vbf\|_2
$$ 
Plugging back in~\eqref{eq:estimateNSPmiddle} yields 
\begin{equation*}
\|\vbf[S^c]\|_{2,1}^{(\omega)} \leq \rho\|\vbf[S^c]\|_{2,1}^{(\omega)} + \tau\|A\vbf\|_2 + \|\zbf\|_{2,1}^{(\omega)} - \|\xbf\|_{2,1}^{(\omega)} + 2 \sigma_s(\xbf)_{2,1}^{(\omega)},
\end{equation*}
or equivalently, for $\rho < 1$ 
\begin{equation}
\label{eq:vsc}
\|\vbf[S^c]\|_{2,1}^{(\omega)} \leq \frac{1}{1-\rho}\left(\|\zbf\|_{2,1}^{(\omega)} - \|\xbf\|_{2,1}^{(\omega)} + \tau\|A\vbf\|_2  + 2 \sigma_s(\xbf)_{2,1}^{(\omega)}\right)
\end{equation}

Putting all the pieces together
\begin{align*}
\|\xbf - \zbf\|_{2,1}^{(\omega)} &= \|\vbf\|_{2,1}^{(\omega)} = \|\vbf[S]\|_{2,1}^{(\omega)} + \|\vbf[S^c]\|_{2,1}^{(\omega)} \\
&\stackrel{\text{Def.~\ref{def:wbrnsp}}}{\leq} \left(1+\rho\right)\|\vbf[S^c]\|_{2,1}^{(\omega)} + \tau\|A\vbf\|_{2} \\ 
&\stackrel{\text{Eq.\eqref{eq:vsc}}}{\leq} \frac{1+\rho}{1-\rho}\left(\|\zbf\|_{2,1}^{(\omega)} - \|\xbf\|_{2,1}^{(\omega)} + 2\sigma_s(\xbf)_{q,p}^{(\omega)}\right) + \frac{2\tau}{1-\rho}\|A(\zbf-\xbf)\|_2.
\end{align*}
\end{proof}

\begin{acorol}
\label{cor:l2pestimates}
Let $B \in \Nbb$ and $\Bcal = (\Bcal_1, \cdots, \Bcal_B)$ be a partition of $\{1, \cdots, N\}$ with the associated weight sequence $\omega = (\omega_1, \cdots, \omega_B)$.
Given $1 \leq p \leq q \leq 2$ suppose that $A \in \Rbb^ {m \times N}$ satisfies the $\ell_\omega^q$-BRNSP of order $s \geq \|\omega\|_\infty^2$ with constants $0 < \rho < 1$ and $\tau > 0$. 
Then for any $\xbf, \zbf \in \Rbb^N$ we have
\begin{equation}
\label{eq:l2pestimates}
\|\zbf - \xbf\|_{2,p}^{(\omega)} \leq \frac{C_\rho}{s^{1-1/p}}\left( \|\zbf\|_{2,1}^{(\omega)} - \|\xbf\|_{2,1}^{(\omega)} + 2\sigma_s(\xbf)_{2,1}^{(\omega)} \right) +  \frac{D_{\rho,\tau}}{s^{1/q-1/p}}\|A(\zbf-\xbf)\|_2,
\end{equation}
with $C_\rho = \frac{(1+\rho)^2}{1-\rho}$ and $D_{\rho,\tau} = \frac{3+\rho}{1-\rho}\tau$.
\end{acorol}

\begin{proof}
The proof follows the one of~\cite[Theorem 4.25]{Foucart13book} with the appropriate changes to accommodate for the weighted block structure. Notice that for any $2q \geq q' \geq q$, using H\"{o}lder's inequality, the $\ell_\omega^{q'}$-BRNSP follows from the $\ell_\omega^q$-BRNSP, namely
\begin{equation}
\label{eq:lqimplieslp}
\|\vbf[S]\|_{2,q'}^{(\omega)} \leq \frac{\rho}{s^{1-1/q'}}\|\vbf[S^c]\|_{2,1}^{(\omega)} + \tau s^{1/q'-1/q}\|A\vbf\|_2
\end{equation}
for all $\vbf \in \Rbb^N$ and $S$ such that $\omega(S) \leq s$. 
Applying Proposition~\ref{prop:lemmaL1NSP} with Equation~\eqref{eq:lqimplieslp} for $q' = 1$ and upon choosing $S$ as the block support of the best weighted $s$ block approximation of $\xbf$ yields
\begin{equation}
\label{eq:lqtol1NSP}
\|\zbf-\xbf\|_{2,1}^{(\omega)} \leq \frac{1+\rho}{1-\rho}\left( \|\zbf\|_{2,1}^{(\omega)} - \|\xbf\|_{2,1}^{(\omega)} + 2\sigma_s(\xbf)_{2,1}^{(\omega)} \right) + \frac{2\tau}{1-\rho}s^{1-1/q}\|A(\zbf-\xbf)\|_2. 
\end{equation}
Now, let $T$ be the block support of the best weighted $s$ block approximation of $\zbf - \xbf$ and apply Equation~\eqref{eq:lqimplieslp} for $q' = q$ together with 
\begin{equation*}
\|\zbf - \xbf\|_{2,p}^{(\omega)} \leq \left\| (\zbf-\xbf)[T^c] \right\|_{2,p}^{(\omega)} + \left\| (\zbf-\xbf)[T] \right\|_{2,p}^{(\omega)} \leq \frac{1}{s^{1-1/p}}\|\zbf-\xbf\|_{2,1}^{(\omega)} + \left\| (\zbf-\xbf)[T] \right\|_{2,p}^{(\omega)}. 
\end{equation*}
to obtain
\begin{equation*}
\|\zbf-\xbf\|_{2,p}^{(\omega)} \leq \frac{1+\rho}{s^{1-1/p}}\|\zbf-\xbf\|_{2,1}^{(\omega)}+\frac{\tau}{s^{1/q-1/p}}\|A(\zbf-\xbf)\|_2. 
\end{equation*}
Finally, the claim follows by plugging in Equation~\eqref{eq:lqtol1NSP}.
\end{proof}

\begin{acorol}
\label{cor:nspBPDN}
For $m,N \in \Nbb$ and \jlb{letting $B \in \Nbb$ denote} the number of blocks in the partition $\Bcal$ of $\{1, \cdots, N\}$. 
In addition, assume given $\omega = (\omega_b)_{b \in \Bcal}$ a sequence of weights with $\omega_i \geq 1$. 
Assume $A \in \Rbb ^{m \times N}$ satisfies the $\ell_\omega^2$-BRNSP of order $s \geq \|\omega\|_\infty^2$ with constant $0 < \rho < 1$ and $\tau > 0$. 
For any $\xbf \in \Rbb^N$ and $\ybf = A \xbf + \ebf$ for some $\|\ebf\|_2 \leq \eta$. 
Let $\widehat{\xbf}$ be the unique solution of 
\begin{equation}
\label{eq:blockBPDN}
\begin{array}{l}
\displaystyle \min_{\zbf \in \Rbb^{N}}\|\zbf\|_{2,1}^{(\omega)} = \sum_{j = 1}^B \omega_j \|\zbf[\Bcal_j]\|_2 \\ 
\text{s. t.~ } \|A\zbf - \ybf\|_2 \leq \eta.
\end{array}
\end{equation}
Then the following error bounds on the reconstruction hold
\begin{align}
\|\xbf - \widehat{\xbf}\|_{2,1}^{(\omega)} &\leq 2C_\rho \sigma_s(\xbf)_{2,1}^{(\omega)} + 2D_{\rho,\tau} \sqrt{s}\eta \label{eq:l21estimates} \\ 
\|\xbf - \widehat{\xbf}\|_2 = \|\xbf - \widehat{\xbf}\|_{2,2}^{(\omega)} &\leq 2C_\rho \frac{\sigma_s(\xbf)_{2,1}^{(\omega)}}{\sqrt{s}} + 2D_{\rho,\tau}\eta, \label{eq:l22estimates}
\end{align}
where the constants $C_\rho$ and $D_{\rho,\tau}$ are the ones from Corollary~\ref{cor:l2pestimates}.
\end{acorol}

\begin{proof}
The proof is obtained by letting $\zbf := \hat{\xbf}$ be the optimal solution to Problem~\eqref{eq:blockBPDN}. 
It then follows by optimality of the solution that $\|\hat{\xbf}\|_{2,1}^{(\omega)} - \|\xbf\|_{2,1}^{(\omega)} \leq 0$. 
Moreover, we have $\|A(\hat{\xbf} - \xbf)\|_2 \leq 2\eta$. 
Hence, by injecting these two estimates, Equation~\eqref{eq:l2pestimates} becomes 
\begin{equation*}
\|\hat{\xbf}-\xbf\|_{2,p} \leq \frac{2C_\rho}{s^{1-1/p}}\sigma_s(\xbf)_{2,1}^{(\omega)} + \frac{2\eta D_{\rho,\tau}}{s^{1/2-1/p}}.
\end{equation*}
Finally, letting $p =1$ and then $2$ concludes the proof. 
\end{proof}

\begin{adef}
Let $A \in \Rbb^{m \times N}$ be a sensing matrix, $\Bcal$ a block structure, and $\omega = (\omega_b)_{b \in \Bcal}$, with $\omega_i \geq 1$, be a sequence of weights.
$A$ is said to have the \emph{weighted block restricted isometry property} of order $s \geq \|\omega\|_\infty$ with constant $\delta \in (0,1)$ (WBRIP($s,\delta$))  with respect to the block structure $\Bcal$ and weights $\omega$ if 
\begin{equation}
\label{eq:blockWRIP}
(1-\delta)\|\xbf\|_2^2 \leq \|A\xbf\|_2^2 \leq (1+\delta)\|\xbf\|_2^2
\end{equation}
holds for every $\xbf$ such that $\|\xbf\|_0^{(\omega)} \leq s$.

The smallest such constant $\delta$ is called the \emph{weighted block restricted isometry constant} and is denoted $\delta_{s}$. 
\end{adef}

\jlb{One should keep in mind that this definition depends on the block-structure considered.} This condition is sufficient for a stable and robust recovery of weighted block sparse signals, as stated in the following theorem.
\begin{atheorem}
\label{thm:recoveryRIP}
Let $A \in \Rbb^{m \times N}$ and $\omega = (\omega_1, \cdots, \omega_B)$ with $\omega_i \geq 1$ be given for a block structure $\Bcal = (\Bcal_1, \cdots, \Bcal_B)$. 
Let $s \geq 2\|\omega\|_\infty^2$ and $\delta_{2s}$ be such that 
\begin{equation}
\label{eq:condDelta2s}
\delta_{2s} < \frac{1}{2\sqrt{2}+1}.
\end{equation}
Assume that the matrix $A$ fulfills the WBRIP($2s,\delta_{2s}$). 
For any $\xbf \in \Rbb^{N}$, let $\ybf = A\xbf + \ebf$ where $\ebf \in \Rbb^m$ is some additive noise such that $\|\ebf\|_2 \leq \eta$.
Then, $\xbf$ can be approximated via the weighted block $\ell^1$ minimization 
\begin{equation*}
\hat{\xbf} := \left\{
\begin{array}{l}
    \displaystyle \argmin_{\zbf \in \Rbb^N} \|\zbf\|_{2,1}^{(\omega)} \\
     \text{ s. t.  } \|A\zbf - \ybf\|_2 \leq \eta
\end{array} \right.
\end{equation*}
with the following error bounds
\begin{align*}
\|\xbf - \hat{\xbf}\|_{2,1}^{(\omega)} &\leq c_\delta \sigma_s(\xbf)_{2,1}^{(\omega)} + d_\delta \sqrt{s} \eta \\
\|\xbf - \hat{\xbf}\|_{2} &\leq \frac{c_\delta}{\sqrt{s}}\sigma_s(\xbf)_{2,1}^{(\omega)} + d_\delta \eta, 
\end{align*}
where the constants $c_\delta := \frac{2(1+\delta_{2s}(2\sqrt{2}-3))^2}{(1-\delta_{2s})(1-\delta_{2s}(2\sqrt{2}+1))}$ and $d_\delta := \frac{2(3+\delta_{2s}(2\sqrt{2}-3))\sqrt{1+\delta_{2s}}}{(1-\delta_{2s})(1-\delta_{2s}(2\sqrt{2}+1))}$ depend only on the RIP constant $\delta_{2s}$.
\end{atheorem}

This theorem is a consequence of the following result
\begin{atheorem}
\label{thm:RIP-NSP}
Let $A \in \Rbb^{m \times N}$ and $\omega = (\omega_1, \cdots, \omega_B)$ with $\omega_i \geq 1$ be given for a block structure $\Bcal = (\Bcal_1, \cdots, \Bcal_B)$. 
Let $s \geq 2\|\omega\|_\infty^2$ and $\delta_{2s}$ satisfies the condition~\eqref{eq:condDelta2s}.
Assume that the matrix $A$ fulfills the WBRIP($2s,\delta_{2s}$). 
Then $A$ fulfills the $\ell^2_\omega-$BRNSP with parameters $\tau = \frac{\sqrt{1+\delta_{2s}}}{1-\delta_{2s}}$ and $\rho = \frac{2\sqrt{2}\delta_{2s}}{1-\delta_{2s}}$.
\end{atheorem}

\begin{proof}
The proof follows the classical ideas from~\cite{Eldar2009group,Rauhut13wCS} with some changes and (local slight) improvements. 

Let $\vbf \in \Rbb^N$ and $S \in \{1,\cdots, B\}$ be such that $\omega(S) \leq s$ and partition $S^c$ into $l$ components $S^c = S_1 \cup S_2 \cup \cdots S_l$ such that (up to the last component $S_l$) all the $\bsupp$ have (almost) the same weighted cardinality: $s - \|\omega\|_\infty^2 \leq \omega(S_i) \leq s$. 
The subsets are ordered according to the non increasing \jlb{rearrangement of $\|\vbf[\Bcal_i]\|_2\omega_i^{-1}$; i.e. $\|\vbf[\Bcal_j]\|_2\omega_j^{-1} \geq \|\vbf[\Bcal_k]\|_2\omega_k^{-1}$} for all $j \in S_i$ and $k \in S_{i+1}$, $1 \leq i \leq l-1$. 
It holds
\begin{align}
\|A(\vbf[S]+\vbf[S_1])\|_2^2 &= \left| \left\langle A(\vbf[S]+\vbf[S_1]), A\vbf-\sum_{i=2}^lA\vbf[S_i] \right\rangle \right| \nonumber \\ 
 &= \left|\left \langle A(\vbf[S]+\vbf[S_1]), A\vbf \right \rangle - \sum_{i = 2}^l\left\langle A(\vbf[S]+\vbf[S_1]), A\vbf[S_i]  \right\rangle\right| \nonumber \\ 
 &\leq \sqrt{1+\delta_{2s}}\left\|\vbf[S\cup S_1]\right\|_2\left\|A\vbf\right\|_2 + \sum_{i=2}^l\left( \left|\langle A\vbf[S], A\vbf[S_i] \rangle \right|  + \left| \langle A\vbf[S_1], A\vbf[S_i] \rangle \right|  \right) \nonumber \\
 &\stackrel{\text{Lemma } \ref{lemma:wbripST}}{\leq} \sqrt{1+\delta_{2s}}\left\|\vbf[S\cup S_1]\right\|_2\left\|A\vbf\right\|_2 + \sum_{i=2}^l \delta_{2s}(\|\vbf[S]\|_2+\|\vbf[S_1]\|_2)\|\vbf[S_i]\|_2 \nonumber \\ 
 \label{eq:rhsRIPtoNSP} &\leq \sqrt{1+\delta_{2s}}\left\|\vbf[S\cup S_1]\right\|_2\left\|A\vbf\right\|_2 + \sqrt{2}\delta_{2s}\|\vbf[S\cup S_1]\|_2 \sum_{i=2}^l\|\vbf[S_i]\|_2.
\end{align}
\jlb{The proof of Lemma~\ref{lemma:wbripST}
can be found in the Appendix~\ref{a:proofOfLemma}.}

On the other hand, using the left hand side of the WBRIP~\eqref{eq:blockWRIP}, we have
\begin{equation*}
\|\vbf[S\cup S_1]\|_2^2 \leq \frac{1}{1-\delta_{2s}}\|A( \vbf[S] + \vbf[S_1])\|_2^2. 
\end{equation*}
Plugging back in~\eqref{eq:rhsRIPtoNSP} and simplifying yields
\begin{equation}
\label{eq:RIPtoNSPalmostThere}
\|\vbf[S \cup S_1]\|_2 \leq \frac{\sqrt{2}\delta_{2s}}{1-\delta_{2s}}\sum_{i=2}^l\|\vbf[S_i]\|_2 + \frac{\sqrt{1+\delta_{2s}}}{1-\delta_{2s}}\|A\vbf\|_2
\end{equation}
Following the proof of Theorem 4.5 from~\cite{Rauhut13wCS}, we can estimate $\sum_{i \geq 2} \|\vbf[S_i]\|_2$. 
For every index $k \in S_i$, $2 \leq i \leq l$, we define $\lambda_k := (\sum_{j \in S_i}\omega_j^2)^{-1}\omega_k^2 \leq \left( s- \|\omega\|_\infty^2 \right)^{-1}\omega_k^2$. 
Noticing that $\sum_{k \in S_i}\lambda_k = 1$, it follows that\jlb{, for any $\ell \in S_i$} $\|\vbf[\Bcal_\ell]\|_2\omega_\ell^{-1} \leq \sum_{k \in S_{i-1}}\lambda_k \|\vbf[\Bcal_k]\|_2\omega_k^{-1} \leq \left( s - \|\omega\|_\infty^2 \right)^{-1}\sum_{k \in S_{i-1}}\|\vbf[\Bcal_k]\|_2\omega_k$. From this we obtain
\begin{align*}
\|\vbf[S_i]\|_2^2 &= \sum_{\ell \in S_i}\|\vbf[\Bcal_\ell]\|_2^2 = \sum_{\ell \in S_i}(\|\vbf[\Bcal_\ell]\|_2 \omega_\ell^{-1})^2\omega_\ell^2 \jlb{\leq \sum_{\ell \in S_i}\left(\left( s - \|\omega\|_\infty^2 \right)^{-1}\sum_{k \in S_{i-1}}\|\vbf[\Bcal_k]\|_2\omega_k\right)^2\omega_\ell^2} \\ 
 &\leq 
\frac{1}{\left(s-\|\omega\|_\infty^2\right)^2}\left( \sum_{\ell \in S_i}\omega_\ell^2\right) \left( \sum_{k \in S_{i-1}} \omega_k \|\vbf[\Bcal_k]\|_2\right)^2 
\end{align*}
\jlb{Taking the square roots and remembering that $\left( \sum_{\ell \in S_i}\omega_\ell^2\right) \leq s$, we arrive at}
$$
\|\vbf[S_i]\|_2 \leq \frac{\sqrt{s}}{s-\|\omega\|_\infty^ 2}\|\vbf[S_{i-1}]\|_{2,1}^{(\omega)} \leq \frac{2}{\sqrt{s}}\|\vbf[S_{i-1}]\|_{2,1}^{(\omega)}.
$$
Finally, plugging back in Eq.~\eqref{eq:RIPtoNSPalmostThere} gives
\begin{align*}
\|\vbf[S]\|_2 &\leq \|\vbf[S] + \vbf[S_1]\|_2 \leq \frac{\sqrt{2}\delta_{2s}}{1-\delta_{2s}}\sum_{i=1}^l\frac{2}{\sqrt{s}}\|\vbf[S_i]\|_{2,1}^{(\omega)} + \frac{\sqrt{1+\delta_{2s}}}{1-\delta_{2s}}\|A\vbf\|_2 \\ 
 &\leq \frac{2\sqrt{2}\delta_{2s}}{\sqrt{s}(1-\delta_{2s})}\|\vbf[S^c]\|_{2,1}^{(\omega)} + \frac{\sqrt{1+\delta_{2s}}}{1-\delta_{2s}}\|A\vbf\|_2,
\end{align*}
which is the $\ell_\omega^2$-BRNSP of order $s$ and constants $\rho = \frac{2\sqrt{2}\delta_{2s}}{1-\delta_{2s}} \in (0,1)$ as soon as Eq.~\eqref{eq:condDelta2s} is fulfilled and $\tau = \frac{\sqrt{1+\delta_{2s}}}{1-\delta_{2s}}$.
\end{proof}

\begin{proof}[Proof of Theorem \ref{thm:recoveryRIP}]
The results follow Corollary~\ref{cor:nspBPDN} with the constants $\rho$ and $\tau$ obtained in Theorem~\ref{thm:RIP-NSP}. 
Indeed, assuming the the matrix $S$ satisfies WBRIP($2s$, $\delta_{2s}$), it follows that it satisfies the $\ell_{\omega}^2$-BRNSP with constants $\rho, \tau$ given in Theorem~\ref{thm:RIP-NSP}. 
Plugging back in Equations~\eqref{eq:l21estimates} and ~\eqref{eq:l22estimates} gives the desired results. 
\end{proof}

\section{Block sparse recovery in practice}
\label{sec:matrices}
We review here some constructions of matrices allowing to use weighted block sparse recovery.

\subsection{Subgaussian random matrices}
Our first example is that of a random matrix. 
Following ideas from~\cite{Foucart13book} and adapting the block sparse recovery from~\cite{Eldar2009group}, we show that a matrix with rows taken as isotropic sub-Gaussian random variables fulfill the a certain WBRIP.

\begin{atheorem}
Let $\tilde{A}$ be an $m \times N$ matrix obtained by stacking $m$ \jlb{isotropic}
independent sub-Gaussian random vectors in $\Rbb^N$ and let $k$ denotes their sub-Gaussian constant. 
\jlb{Define $A = \tilde{A}/\sqrt{m}$.}
Let $\Bcal$ be a block structure with associated weights $(\omega_i)_{i \in \Bcal}, \omega_i \geq 1$.
Let $s \geq \|\omega\|_\infty^2$ and $\delta \in (0,1)$. 
Then $A$ satisfies the $WBRIP(s,\delta)$ with probability $1-\varepsilon$, provided 
\begin{equation}
\label{eq:m4RIP}
m \geq Ck^4s\delta^{-2}\ln\left(\frac{e \kappa}{s}\right) + Ck^4\delta^{-2}\ln\left( \frac{2}{\varepsilon} \right).
\end{equation}
Here $\kappa$ denotes the maximum number of blocks of the smallest size in the block structure. 
\end{atheorem}

\begin{proof}
Let us first introduce the set of admissible block supports (for a given weighted sparsity $s$ and sequence of weights $(\omega_i)_{i \in \Bcal}$)
\begin{equation*}
\Ical := \{I \subseteq \Bcal: \omega(I) = \sum_{i \in I} \omega_i^2 \leq s\}.
\end{equation*}
Since $\omega_i \geq 1$ for all $i$, it follows that $|I| \leq s$ for all $I \in \Ical$. A counting argument yields 
\begin{equation}
\label{eq_nbSupports}
|\Ical| \leq { \lceil N/d_\text{min} \rceil \choose s } = { \kappa \choose s} , 
\end{equation}
where $d_\text{min} := \min_{b \in \Bcal} d_b$ and $\kappa$ denotes the maximum number of blocks of the smallest possible size (note that here $\kappa$ is quite an overestimation of $B$, the number of blocks in the block structure.). 

The goal of the proof is to bound the minimum and maximum singular values of the restrictions of $A$ to block support $S \in \Ical$. 
To this end, define 
\begin{align*}
\sigma_m &:= \min_{S \in \Ical} \sigma_{\text{min}}(A_S), \\
\sigma_M &:= \max_{S \in \Ical} \sigma_{\text{max}}(A_S),
\end{align*}
where we let $A_S$ denote the matrix defined by extracting the columns from $A$ supported on the blocks of $S$. 

Following the definition of the WBRIP (see Eq.~\eqref{eq:blockWRIP}), we have that 
\begin{equation}
\label{eq:RIPsingVal}
1-\delta \leq \sigma_m \leq \sigma_M \leq 1 + \delta, 
\end{equation}
for a certain $\delta \in (0,1)$. 

We now recall a result from~\cite{HDP18} about the minimum and maximum singular values of an isotropic random matrix
\begin{atheorem}
\label{thm:HDP}
Let $A$ be an $m \times n$ matrix whose rows $A_i$ are independent, mean-zero, sub-Gaussian isotropic random vectors in $\Rbb^n$. 
Then for any $t \geq 0$ we have 
\begin{align*}
\Pbb(\sigma_\text{max}(A/\sqrt{m}) > 1+ ck^2\sqrt{\frac{n}{m}} + t) &\leq e^{-\frac{mt^2}{c^2k^4} } \\
\Pbb(\sigma_\text{max}(A/\sqrt{m}) < 1 - ck^2\sqrt{\frac{n}{m}} - t) &\leq e^{-\frac{mt^2}{c^2k^4} }.
\end{align*}
where $k$ denotes the $\psi_2$ norm of the random vectors \footnote{The $\psi_2$ norm is related to the sub-Gaussian constant, in other words, how close to a Gaussian the random variable is. The interested reader may read more in the book~\cite{HDP18}. }
\end{atheorem}
Hence, Eq.~\eqref{eq:RIPsingVal} is violated, if it is violated for at least one of the support in $\Ical$. 
A union bound combined with Theorem~\ref{thm:HDP} gives the upper bound. 
To this hand, notice that by taking the number of measurements large enough (by increasing the constant in Eq.~\eqref{eq:m4RIP}), we can make \jlb{$ck^2\sqrt{\frac{n}{m}} \leq \delta/2$} and letting $t = \delta / 2$ yields
\begin{align*}
\Pbb(\sigma_M > 1+\delta) &\leq \sum_{I \in \Ical} \Pbb(\sigma_{\text{max}}(A_I) > 1+ \delta ) \leq {\kappa \choose s}e^{\frac{-m\delta^2}{c'k^4}}.
\end{align*}
Together with the well known bound for binomial coefficients
$$
{\kappa \choose s} \leq \left( \frac{e\kappa}{s} \right)^s,
$$
we arrive at
\begin{equation*}
\Pbb(\sigma_M > 1 + \delta) \leq e^{s\ln(e\kappa/s)-\frac{m\delta^2}{c'k^4}}.
\end{equation*}
One gets the same results for the lower bound on the smallest singular value. 
Hence, the singular values are within $[1-\delta, 1+\delta]$ with probability higher than $1-\varepsilon$ provided
\begin{align*}
s\ln(e\kappa/s)- \frac{m\delta^2}{c'k^4} \leq \ln(\varepsilon/2).
\end{align*}
Solving for $m$ gives us the desired result.

\end{proof}

\underline{Notes:}
\begin{itemize}
\item This result generalizes the results of~\cite{Eldar2009group} in two ways: it considers random variables that are not only Gaussian, but also sub-Gaussian and even those that may have non-independent entries. 
Moreover, the use of weights allows for more flexibility in the model. 
Keep in mind that setting all weights uniformly equal to 1 gives the result of~\cite{Eldar2009group} for Gaussian random variables.
\item The use of isotropic sub-Gaussian rows is important to generate interest in such a model. 
Indeed, the whole idea of weighted compressed sensing is to recover vectors for which we have prior knowledge / belief about how the support should look like. 
What this means will become clearer in the section below. 
If we dealt only with Gaussian (or in general for entries i.i.d.  according to a unique probability distribution) random variables, those weights would not bare any meaning.
\end{itemize}

For the purpose of comparison, assume that the blocks all have the same size. 
In this case, $\kappa = N/d = B := |\Bcal|$. 
Let $k \leq s \times \kappa$ be the true sparsity of the signal (measured independently of the block structure). 
Our results tell that a number of measurements scaling as $m \geq C \frac{k}{\kappa}\ln(e k)$, which is an improvement over results for traditional compressed sensing. 

\subsection{Joint sparsity and weighted sparse recovery}
Continuing the discussion above, in which all the blocks have the same dimension, we may consider the case of joint sparsity in a multiple measurements setup. 

\subsubsection{RIP for tensor product matrices}

Here, we have $d$ vectors, each of size $B$ which are measured by a single design matrix $A \in \Rbb^{m \times B}$. 
This problem can be modeled as 
$$
Y = A X + \Theta,
$$
where $X = [\xbf_1, \cdots, \xbf_d] \in \Rbb^{B \times d}$ and similarly $Y = [\ybf_1, \cdots, \ybf_d]$ with $\ybf_i = A \xbf_i$ plus some noise, which is contained in the matrix $\Theta$. 
The term \textit{joint sparsity} refers to the case where all the vectors $\xbf_i$ share the same sparsity pattern. 
Letting $S = \cup_{i=1}^d\operatorname{supp}(\xbf_i)$, we say that the matrix $X$ has a joint weighted sparsity of order $s$ (with respect to a sequence of weights $\omega_1, \cdots, \omega_B)$ if $\sum_{i \in S} \omega_i^2 \leq s$. 

This is an example of a block sparse model in the sense that, vectorizing every matrices row wise, yields (forgetting the noise for now)
\begin{equation*}
\wY = \wA\wX,
\end{equation*}
where $\widetilde{Y}$ (resp. $\widetilde{X}$) contains all the rows of $Y$ (resp. $X$) stacked one above the other in a column format and $\widetilde{A} = A \otimes I$ is the tensor product of $A$ and the identity matrix.

The next theorem generalizes results from~\cite{Eldar2009group}: 
\begin{atheorem}
\label{thm:equivalenceFrobTrad}
Let $A \in \Rbb^{m \times B}$ and let $\widetilde{A}= A \otimes I_d$. 
$A$ satisfies the $WRIP(s,\delta)$ if and only if $\widetilde{A}$ satisfies $WBRIP(s,\delta)$.
\end{atheorem}

This result proves two things. 1) that we will not get more recovery matrices by using the joint sparsity structure but 2) all matrices used for the recovery of a single vector can be used (after a tensor product) for the recovery of joint sparse vectors. 
In particular, in the example we derive afterwards in the context of high-dimensional parametric PDEs, we may use the sensing matrices used for orthonormal systems to approximate the solution in a polynomial chaos expansion. 

\begin{proof}
Remember that the traditional $WRIP(s,\delta)$ of a matrix $A$ reads 
\begin{equation}
\label{eq:tradRIP}
(1-\delta)\|\xbf\|_2^2 \leq \|A\xbf\|_2^2 \leq (1+\delta)\|\xbf\|_2^2, \quad \text{ for all weighted $s$-sparse } \xbf. 
\end{equation}
Before we show the equivalence, let us rephrase the WBRIP and WRIP for this particular tensor product structure. 
Note that for any vector $\widetilde{X} \in \Rbb^N$ vectorized from a matrix $X \in \Rbb^{B \times D}$, $BD = N$, we have 
\begin{equation*}
\|\widetilde{A}\widetilde{X}\|_2^2 = \operatorname{Tr}(X^TA^TAX) = \|AX\|_\text{Frob}^2 \quad \text{ and } \quad \|\widetilde{X}\|_2^2 = \|X\|_\text{Frob}^2. 
\end{equation*}
Hence the WBRIP is equivalent to 
\begin{equation}
\label{eq:FrobRIP}
(1-\delta)\|X\|_\text{Frob}^2 \leq \|AX\|_\text{Frob}^2 \leq (1+\delta)\|X\|_\text{Frob}^2\quad \text{ for all $s$-weighted joint sparse } X.
\end{equation}

Assume~\eqref{eq:FrobRIP} holds. Let $\xbf \in \Rbb^B$ be a weighted $s$ sparse vector. 
And let $X$ be the matrix with $d$ copies of $\xbf$. 
Then $\|X\|_{Frob}^2 = d \|\xbf\|_2^2$ and $\|AX\|_\text{Frob}^2 = d\|A\xbf\|_2^2$. 
$X$ being weighted joint sparse of order $s$ and simplifying by $d$, \eqref{eq:tradRIP} follows. 

Conversely, assume that~\eqref{eq:tradRIP} holds and let $\xbf_1, \cdots, \xbf_d$ be $d$ $B$-dimensional vectors which all are weighted $s$-sparse, with the same sparsity pattern. 
It holds
$$
(1-\delta)\|\xbf_i\|_2^2 \leq \|A\xbf_i\|_2^2 \leq (1+\delta)\|\xbf_i\|_2^2, \quad \text{for all } 1\leq i \leq d.
$$
Summing for all $i$ results in 
$$
(1-\delta)\sum_{i=1}^d\|\xbf_i\|_2^2 = (1-\delta)\|X\|_\text{Frob}^2 \leq \sum_{i=1}^d\|A\xbf_i\|_2^2 = \|AX\|_\text{Frob}^2 \leq (1+\delta)\sum_{i=1}^d\|\xbf_i\|_2^2 = (1+\delta)\|X\|_\text{Frob}^2.
$$
\end{proof}

A consequence is the following theorem, which we require in the next section
\begin{atheorem}
\label{thm:wbripBOS}
For some parameters $\delta, \gamma \in (0,1)$,
let $(\psi_j)_{j \in \Lambda}$ with $|\Lambda| = B$ a finite dimensional orthonomal basis with orthogonalization measure $\eta$. 
Let $\omega = (\omega_i)_{1 \leq i \leq B}$ with $\omega_i \geq \|\psi_j\|_\infty$ be a sequence of weights. 
Let 
\begin{equation*}
m \geq C \delta^{-2} s \max\{\log^3(s)\log(B), \log(1/\gamma)\}.
\end{equation*}
Given $m$ samples $\ybs^{(i)}$, drawn independently at random from $\eta$, let $A \in \Cbb^{m \times B}$ be the design matrix whose entries are defined as $a_{ij} = \frac{1}{\sqrt{m}}\psi_j(\ybs^{(i)})$. 

Then $A \otimes I_J$ fulfills the weighted joint sparse RIP defined in~\eqref{eq:FrobRIP}.
\end{atheorem}

The proof is simply the result of Theorem~\eqref{thm:equivalenceFrobTrad} applied to matrices from orthonormal systems fulfilling the weighted RIP, see ~\cite{Rauhut13wCS}. 

\subsection{Parametric function approximation}

In this section, we analyze the use of joint sparsity for the approximation of parametric functions. 
A parametric function is defined as 
\begin{equation*}
f: \begin{array}{ccc} \Omega \times \Ucal  &\to &\Rbb \\ (x, \ybs) &\mapsto &f(x;\ybs).\end{array}
\end{equation*}
Here $\ybs$ denotes a vector of parameters while $\Omega \subset \Rbb^{n}$ with usually $n \in \{1,2,3\}$ denotes the spatial domain for the spatial variable $x$. 
We want $\Ucal = [-1,1]^d$ to be high-dimensional (up to countably many parameters)\footnote{Note the unorthodox semi-colon notation to emphasize the difference in roles played by $x$ and $\ybs$.}. 

\begin{assumption}[Decoupling] 
\label{assump:decoupling}
The parameter and spatial coordinates are decoupled. 
This implies that $f$ lives in a tensor product space $f \in \Xcal \otimes P$. 
Throughout, $\Xcal$ will denote a Hilbert space, associated to the norm $\|\cdot\|_\Xcal$, and inner product $\langle \cdot, \cdot \rangle_\Xcal$ (we drop the subscript when it is clear from the context). 
\end{assumption}

\begin{assumption}[Tensorized polynomials] 
\label{assump:tensorPoly}
The parameter space is itself a tensor space and is spanned by (tensorized) polynomials. 
Hence, given an orthonormal basis of polynomials $(T_i)_{i \in \Nbb_0}$ for $[-1,1]$, we may write 
$P \subset \operatorname{span}\{T_\nu: \nu \in \Fcal\}$ where 
\begin{equation}
\label{eq:fcal}
\Fcal := \{ \nu \in \Nbb_0^d: \|\nu\|_0 := \supp(\nu) < \infty \}
\end{equation} is the set of multi-indices with finite support and $T_\nu(\ybs) = \prod_{i \in \operatorname{supp}(\nu)}T_{\nu_i}(y_i)$ corresponds to the tensorized polynomial with (multi)index $\nu \in \Fcal$.
We will for now assume that only finitely many polynomials are needed; i.e. there exists $\Lambda \subset \Fcal$ with $|\Lambda| = N < \infty$ such that $\Ycal = \operatorname{span}\{T_\nu: \nu \in \Lambda\}$. 
\jlb{We assume moreover this family of polynomial is orthonormal with respect to a measure $\eta$ and denote the associated inner product as $\langle \cdot, \cdot \rangle$.}
\end{assumption}

Since $\Xcal$ is a Hilbert space, it enjoys a countable family as an orthonormal basis $(\varphi_j)_{j \in \Jcal }$. 
All this together allows us to write the decoupling equation
\begin{equation*}
f(x;\ybs) = \sum_{\nu \in \Lambda, j \in \Jcal} f_\nu^j \varphi_j(x) T_\nu(\ybs). 
\end{equation*}

The goal is to find an approximation $f^\#$ which is \textit{close enough} to $f$ in a Bochner sense, for $p \geq 1$:
\begin{equation}
\label{eq:BochnerNorm}
\|f-f^\#\|_{L^p(\Xcal; \Ucal, \eta)}^p = \int \limits_{\ybs \in \Ucal} \|f(\cdot;\ybs) - f^\#(\cdot;\ybs)\|_\Xcal^p\dif{\eta}(\ybs)
\end{equation}
with the classical adaptation via $\operatorname{ess-sup}$ in the case $p = \infty$. 

Given the assumptions above, this can be easily recast to the problem of recovering a joint-sparse matrix as in the previous section. 
This matrix can in theory be infinite, as it is indexed over $j \in \Jcal$. 
In applications, we deal with finite subspaces of the Hilbert space $\Xcal$. 
Moreover, since $(\varphi_j)_{j \in \Jcal}$ is assumed to be an orthonormal basis, its associated pair of analysis / synthesis operators $S: \Xcal \to \ell^2(\Jcal), S^*: \ell^2(\Jcal) \to \Xcal$ are \jlb{isometries}
\begin{equation*}
    S: \left\{\begin{array}{ccl}
        \Xcal &\to &\ell^2(\Jcal)  \\
        v &\mapsto &\vbf := Sv = \left( \langle v, \varphi_j \rangle_\Xcal \right)_{j \in \Jcal} 
    \end{array} \right. \, \text{ and } 
    S^*: \left\{\begin{array}{ccl} 
        \ell^2(\Jcal) &\to &\Xcal   \\
        \vbf := \left(v_j\right)_{j \in \Jcal} &\mapsto &v := S^*\vbf = \sum_{j \in \Jcal}v_j \varphi_j 
    \end{array} \right.
\end{equation*}
In particular, such mappings preserve the inner products and satisfy Parseval's identity, for all $v, w \in \Xcal$, 
\begin{align}
    \langle v, w \rangle_\Xcal &= \langle Sv, Sw\rangle_{\ell^2(\Jcal)}, \nonumber \\
    \label{eq:parseval} \|v\|_\Xcal &= \|Sv\|_{\ell^2(\Jcal)}.
\end{align}

Two values for $p$ in~\eqref{eq:BochnerNorm} are particularly important: $p=2$ and $p=\infty$. 
Indeed\footnote{One could arrive at the same result and spare some time reading the derivation by introducing the inner product induced by the tensor product of two Hilbert spaces ($\Xcal$ and $\Ucal$). We choose not to use this --valid-- abstract setting.} 	, for $p=2$,
\begin{align}
\|f-f^\#\|_{L^2}^2 &= \int \limits_{\ybs \in \Ucal} \|f(\ybs) - f^\#(\ybs)\|_\Xcal^2 \dif{\eta}(\ybs)  \nonumber \\
									 &= \int \limits_{\ybs \in \Ucal} \left\langle f(\ybs) - {f^\#}(\ybs), f(\ybs) - {f^\#}(\ybs) \right \rangle_{\Xcal} \dif{\eta}(\ybs) \nonumber \\
									&= \int \limits_{\ybs \in \Ucal} \left\langle \sum_{\nu \in \Lambda} \left(f_\nu - f^\#_\nu\right) T_\nu(\ybs), \sum_{\mu \in \Lambda} \left(f_\mu - f^\#_\mu \right) T_\mu(\ybs)\right\rangle_\Xcal \dif{\eta}(\ybs) \nonumber \\
									&= \int \limits_{\ybs \in \Ucal} \sum_{\nu \in \Lambda} \sum_{\mu \in \Lambda} T_\nu(\ybs)T_\mu(\ybs)\langle f_\nu - f^\#_\nu, f_\mu - f^\#_\mu\rangle_\Xcal \dif{\eta}(\ybs) \nonumber \\
									& = \sum_{\nu \in \Lambda} \sum_{\mu \in \Lambda} \langle f_\nu - f^\#_\nu, f_\mu - f^\#_\mu\rangle_\Xcal \int \limits_{\ybs \in \Ucal} T_\nu(\ybs)T_\mu(\ybs) \dif{\eta}(\ybs) \nonumber \\
									&= \sum_{\nu \in \Lambda}\sum_{\mu \in \Lambda} \langle f_\nu - f^\#_\nu, f_\mu - f^\#_\mu\rangle_\Xcal \langle T_\nu, T_\mu \rangle_\eta \nonumber \\
									&= \sum_{\nu \in \Lambda} \| f_\nu - f^\#_\mu\|_\Xcal^2 =: \|f-f^\#\|_{\Xcal,2}^2. \label{eq:approxBochner2}
\end{align}
If $F = (Sf_\nu)_{\nu \in \Lambda}$ corresponds to the matrix of coefficients: $F_{ij} = f_i^j, i \in \Lambda, j \in \Jcal$, the last equality is given by $ \|F-F^\#\|_\text{Frob}^2$. 
Note that we have abused some notations: $L^2$ as a shorthand for the Bochner norm, dropping the $x$ or $\cdot$ in the mappings $f$ evaluated at parameter vectors $\ybs$. 
For $p = \infty$: 
\begin{align}
\|f-f^\#\|_{L^\infty} &= \|\sum_{\nu \in \Lambda}\left(f_\nu-f^\#_\nu\right)T_\nu\|_{L^\infty} = \max_{\ybs \in \Ucal} \| \sum_{\nu \in \Lambda} \left( f_\nu-f^\#_\nu \right)T_\nu(\ybs)\|_\Xcal \nonumber \\
											&\leq \sum_{\nu \in \Lambda} \|f_\nu - f^\#_\nu\|_\Xcal\|T_\nu\|_\infty \nonumber \\
											&\leq \sum_{\nu \in \Lambda} \|f_\nu - f^\#_\nu \|_{\Xcal} \omega_\nu = \|f-f^\#\|_{\Xcal,1}^{(\omega)} = \|F-F^\#\|_{2,1}^{(\omega)}, \label{eq:approxBochnerInf}
\end{align}
assuming we have chosen a weight sequence $(\omega_\nu)_{\nu \in \Lambda}$ such that $\|T_\nu\|_\infty \leq \omega_\nu$. 
Eqs.~\eqref{eq:approxBochner2} and~\eqref{eq:approxBochnerInf} tell us that the approximation of the function can be estimated using mixed norms, as introduced for the block or joint sparse model. 
These norms have been extended to the case of blocks being vectors in function spaces, with the associated norms. 
More precisely, using sparse approximation techniques, Theorem~\ref{thm:recoveryRIP} gives us bounds on the recovery when using block-sparse minimization. 
Assuming that our target (parametric) function $f$ is well approximated by block/joint-sparse vectors, then Theorem~\ref{thm:recoveryRIP} indeed gives an accuracy of the approximation and a number of function evaluations required. 
This is the main point of our next theorem.

\begin{atheorem}
\label{thm:fctApprox}
Let $(T_\nu)_{\nu \in \Lambda}$, $|\Lambda| = B < \infty$ be a finite orthonormal system with orthogonalization measure $\eta$. 
Let $\omega = (\omega_\nu)_{\nu \in \Lambda}$, with $\omega_\nu \geq \|T_\nu\|_\infty$, be a sequence of weights. 
Let $s \geq 2 \|\omega\|_\infty^2$ and $\gamma \in (0,1)$. Let
$$
m \geq C s\max\{\log^3(s)\log(B), \log(\gamma^{-1})\}
$$ 
and draw $m$ samples $\ybs^{(i)}$, $1 \leq i \leq m$ independently at random from $\eta$. 
Furthermore, let $A$ be the sensing matrix obtained by evaluating the basis functions at the given samples; i.e. $A_{i,\nu} = T_\nu(\ybs^{(i)})$. 

Then, with probability at least $1-\gamma$, the following holds for all functions $f = \sum_{\nu \in \Lambda} f_\nu T_\nu \in \Xcal \otimes P_\Lambda$ with $f_\nu \in \Xcal = \operatorname{span}\{\varphi_j: j \in \Jcal\}$, where $\left( \varphi_j \right)_{j \in \Jcal}$ is a finite orthonormal system. 
For $1 \leq i \leq m$, let $Y_i^T := (\langle f(\ybs^{(i)}), \varphi_j \rangle)_{j \in \Jcal} + \varepsilon_i \in \Rbb^{1 \times \Jcal}$ be the (noisy) coefficients of the target function evaluated at $\ybs^{(i)}$, with $\sqrt{\sum_i \|\varepsilon_i\|_2^2} \leq \varepsilon$. 
Let $\wF$ be the solution of 
\begin{equation*}
\min_{Z \in \Rbb^{\Lambda \times \Jcal}}\|Z\|_{2,1}^{(\omega)}, \quad \text{ subject to } \|AZ-Y\|_\text{Frob} \leq \varepsilon.
\end{equation*}
and consider $\wf(x;\ybs) := \sum_{\nu \in \Lambda}\sum_{j \in \Jcal} \wF_\nu^j T_\nu(\ybs)\varphi_j(x)$. 
Then 
\begin{align*}
\|f-\wf\|_{L^\infty} \leq \|F-\wF\|_{2,1}^{(\omega)} &\leq  c\sigma_s(F)_{2,1}^{(\omega)} + \frac{d}{\sqrt{m}}\varepsilon,\\
\|f-\wf\|_{L^2} = \|F-\wF\|_{2,2}^{(\omega)} &\leq \frac{c'}{\sqrt{s}}\sigma_s(F)_{2,1}^{(\omega)} + \frac{d'}{\sqrt{m}}\varepsilon,
\end{align*}
for some universal constants $c,d$ and $c',d'$. 
\end{atheorem}

We will in fact prove a stronger theorem, which works in (potentially) infinite dimensional spaces $\Xcal$, without mentioning a specific basis.

\begin{atheorem}
\label{thm:fctApproxInf}
Let $(T_\nu)_{\nu \in \Lambda}$, $|\Lambda| = B < \infty$ be a finite orthonormal system with orthogonalization measure $\eta$. 
Let $\omega = (\omega_\nu)_{\nu \in \Lambda}$ with $\omega_\nu \geq \|T_\nu\|_\infty$ be a sequence of weights. 
Let $s \geq 2 \|\omega\|_\infty^2$ and $\gamma \in (0,1)$. Let
$$
m \geq C s\max\{\log^3(s)\log(B), \log(\gamma^{-1})\}
$$ 
and draw $m$ samples $\ybs^{(i)}$, $1 \leq i \leq m$ independently at random from $\eta$. 
Furthermore, let $A$ be the sensing matrix obtained by evaluating the basis functions at the given samples; i.e. $A_{i,\nu} = T_\nu(\ybs^{(i)})$. 

Then with probability at least $1-\gamma$, the following holds for all functions $f = \sum_{\nu \in \Lambda} f_\nu T_\nu \in \Xcal \otimes P_\Lambda$ with $f_\nu \in \Xcal$. 
For $1 \leq i \leq m$, let 
$b^{(i)}$ approximates $f_\Lambda(\ybs^{(i)}) := \sum_{\nu \in \Lambda} f_\nu T_\nu(\ybs^{(i)})$ with accuracy $\varepsilon$. 
Let $\wF$ be the solution of 
\begin{equation*}
\min_{z \in \Xcal \otimes P_\Lambda}\|z\|_{\Xcal,1}^{(\omega)}, \quad \text{ subject to } \|Az-b\|_\text{2} \leq \sqrt{m}\varepsilon.
\end{equation*}
and consider $\wf(x;\ybs) := \sum_{\nu \in \Lambda}\wF_\nu(x) T_\nu(\ybs)$. 
Then 
\begin{align*}
\|f-\wf\|_{L^\infty} \leq \|f-\wf\|_{\Xcal,1}^{(\omega)} &\leq  c\sigma_s(f)_{\Xcal,1}^{(\omega)} + d\sqrt{s}\varepsilon,\\
\|f-\wf\|_{L^2} = \|f-\wf\|_{\Xcal,2}^{(\omega)} &\leq \frac{c'}{\sqrt{s}}\sigma_s(f)_{\Xcal,1}^{(\omega)} + d'\varepsilon,
\end{align*}
for some universal constants. 
\end{atheorem}

\begin{proof}
Given the number of measurements, the matrix $A/\sqrt{m}$ fulfills the WBRIP, according to Thm.~\ref{thm:wbripBOS}. 
This implies the approximation bounds by applying Theorem~\ref{thm:recoveryRIP} to the matrix of coefficients. 
Finally, the computations above this theorem yield the bounds on the error of the functions in terms of norms in Bochner spaces. 
\end{proof}

\begin{rmk}
\label{rmk:}
Theorem~\ref{thm:fctApprox} is obtained by applying Theorem~\ref{thm:fctApproxInf} and Parseval identity~\eqref{eq:parseval} so that the norms obtained on $\Xcal$ are transferred to an $\ell^2$ norm on the coefficients.
\end{rmk}

\section{Application: Approximation of high-dimensional parametric PDEs}
\label{sec:PDEsExamples}

\jlb{This section extends results for the numerical approximation of high-dimensional parametric elliptic PDEs. 
Using a compressed sensing approach,~\cite{Rauhut14CSPG,Bouchot15SLCSPG} were interested in computing \emph{quantities of interest} of a parametric solution. 
Namely~\cite{Rauhut14CSPG} introduced the method from a theoretical point of view, while~\cite{Bouchot15SLCSPG} investigates the numerical applicability of the method, approximating solutions to a diffusion problem in relatively high (parametric) dimensions; even infinite, when considering a truncated operator, as is explained later on in Eqs.~\eqref{eq:decayEnergyOperator} and ~\eqref{eq:weakTruncated}.
Here, we want to approximate the full-solution, point-wise in the spatial coordinate $\xbf$, uniformly for all parameter $\ybs$.
}
Note that an improved computational complexity may be obtained by a multi-level approach~\cite{Bouchot16MLCSPG}, but this is left aside to avoid overcomplicating the notations and exposition. 
While in the previous Section we were considering any polynomial basis for the space of parameters, we will from now on specialize our results to the \Tscheb system.

Let us first recall that the univariate \Tscheb polynomials form an orthogonal family with respect to the orthogonolization measure 
$$
\dif{\eta_1}(t) = \frac{\dif{t}}{\pi\sqrt{1-t^2}}
$$
and defined as 
$$
T_j(t) = \sqrt{2}\cos\left(j \arccos(t)\right), \quad j \geq 1 \qquad T_0 \equiv 1.
$$
We are dealing here with multivariate polynomial. Define, for $\ybs \in \Ucal$
\begin{equation*}
\dif{\eta}(\ybs) := \bigotimes_{j \geq 1}\dif{\eta_1}(y_j) 
\end{equation*}
the tensorized orthogonalization measure and, for $\nu \in \Fcal$
\begin{equation*}
T_\nu(\ybs) = \prod_{j: \nu_j \neq 0}T_{\nu_j}(y_j).
\end{equation*}
The set of tensorized \Tscheb polynomials is orthonormal with respect to $\dif{\eta}$. 
Moreover, it holds
\begin{equation}
\label{eq:linfTscheb}
\|T_\nu\|_\infty = 2^{\|\nu\|_0/2}.
\end{equation}

\subsection{Affine parametric operator equations}

We consider a family of operator equations
\begin{equation}
\label{eq:affineFamily}
A(\ybs): \Xcal \to \Ycal', \quad A(\ybs) = A_0 + \sum_{j \geq 1}y_jA_j
\end{equation}
and try to (numerically) approximate 
$u(\cdot;\ybs) \in \Xcal$ such that $A(\ybs)u(\ybs) = f$, uniformly for all parameter $\ybs$. 
Should the operator $A(\ybs)$ be invertible for all $\ybs$, this problem accounts to stably inverting the family of operators. 
While the linear dependence may be seen as a strong limitation, one should remember that such decompositions may be obtained for instance by variance decomposition methods of stochastics fields such as the Karhunen Lo\`{e}ve decomposition, see~\cite{cohen2011analytic,SchTodor06}. 
To generalize even more, the only requirement for this work to apply, is that the solution be analytic with respect to any finite subset of the parameters; the linear dependence is just one example where this is the case.
Throughout, we will assume the \textit{mean field} $A_0$ to be invertible.

A prototypical example is given by the Poisson problem in divergence form
\begin{equation*}
-\operatorname{div} (a(\cdot; \ybs) \nabla u) = f, \text{ in } \Omega
\end{equation*}
set on a Lipschitz bounded domain $\Omega \subset \Rbb^ n$ for a (spatial) dimension $n \in \Nbb$ (typically $n \in \{ 1,2,3\}$)
with Dirichlet boundary conditions $u_{|\partial \Omega} \equiv 0$. 
We assume that the diffusion coefficient has an affine-linear dependence on the parameter $\ybs$, in the sense that there exists $(\psi_j)_{j \in \Nbb}$ such that
\begin{equation*}
a(x;\ybs) = \bar{a}(x) + \sum_{j \in \Nbb}y_j \psi_j(x)
\end{equation*}
for every $x \in \Omega$ and $j \in \Ucal = [-1,1]^\Nbb$. 
The functions $\bar{a}$ and $\psi_j$ are taken in $L^\infty(\Omega)$. 
Here the solution space is given by $\Xcal = \Ycal = H_0^1(\Omega)$ with the usual inner product; whence we consider $f \in \Ycal' = \Xcal' = H^ {-1}(\Omega)$.

\subsubsection{Existence of solutions}
It is now mathematical folklore that solutions to~\eqref{eq:affineFamily} exist whenever the family of operators fulfills a certain uniform ellipticity assumption (UEA). 
For the purpose of our work, we require (and express only) a somewhat stronger version of this assumption, written in a weighted form, with weights $(v_j)_{j \geq 1}$:
\begin{equation}
\label{eq:wuea}
\tag{$wUEA$}
\sum_{j \geq 1} v_j^{(2-p)/p}b_{0,j} \leq \kappa_{\vbf,p} < 1, \quad b_{0,j} := \|A_0^{-1}A_j\|_{\Lcal(\Xcal,\Xcal)}.
\end{equation}

Let us first introduce some properties. Namely, we say that the family of operators $A(\ybs)$ is $\inf-\sup$-stable if there exists a $\mu$ such that 
\begin{equation}
\label{eq:infsupContinuous}
\begin{array}{c}
\inf_{0 \neq v \in \Xcal} \sup_{0 \neq w \in \Ycal} \frac{\left\langle A(\ybs)v,w \right\rangle}{\|v\|_\Xcal\|w\|_\Ycal} \geq \mu > 0, \\
\inf_{0 \neq w \in \Ycal} \sup_{0 \neq v \in \Xcal} \frac{\left\langle A(\ybs)v,w \right\rangle}{\|v\|_\Xcal\|w\|_\Ycal} \geq \mu > 0,
\end{array}
\end{equation}
where we have used the duality bracket notation.

\begin{aprop}
Assume $A_0$ is $\inf-\sup$ stable (i.e.~\eqref{eq:infsupContinuous} holds at least for $\ybs = 0$, with constant $\mu_0 > 0$) and assume that~\eqref{eq:wuea} holds for some \jlb{weights $v_j \geq 1$, $j \geq 1$}. 
Then for any $\ybs \in \Ucal$, the weak solution $u(\ybs)$ exists and is unique and satisfies the uniform estimates
\begin{equation*}
\sup_{\ybs \in \Ucal} \|u(\ybs)\|_\Xcal \leq \frac{1}{\mu_0(1-\kappa_{\vbf,p})}\|f\|_{\Ycal'}.
\end{equation*}
\end{aprop}

\begin{proof}
The $\inf-\sup$ conditions are equivalent to the bounded invertibility of the operator. 
Moreover, since the $\inf-\sup$ conditions hold for the mean field $A_0$, the invertibility of the operator $A(\ybs)$ is done by considering a Neumann perturbation argument; i.e. writing $A(\ybs) = A_0(I + \sum_{j \geq 1}y_jA_0^{-1}A_j)$, we see that 
$$
\sum_{j \geq 1} b_{0,j} \leq \sum_{j \geq 1} b_{0,j}v_j^{(2-p)/p} \leq \kappa_{\vbf,p} < 1.
$$
The result follows. 
\end{proof}

\subsubsection{Main result: Approximation of high-dimensional parametric PDEs}

The problem described above having indeed (a family of) solutions, we may approximate them numerically, uniformly for all parameters.

\begin{atheorem}
\label{thm:EDPrecovery}
Let $A$ be an affine family of operators as defined in~\eqref{eq:affineFamily} and let $u$ be the solution to $A(\ybs)u = f$.
Assume that the $\inf-\sup$ conditions~\eqref{eq:infsupContinuous} and~\eqref{eq:infsupDiscrete} hold for the mean-field $A_0$. 
Let $p \in (0,1)$ and assume that~\eqref{eq:wuea} hold for a given weight sequence $\vbf = (v_j)_{j \in \Nbb}$, with $v_j \geq 1$, for all $j \in \Nbb$. 
Assume moreover that the sequence of operators are compressible, in the sense that 
\begin{equation}
\label{eq:lpSummability}
\sum_{j \in \Nbb} v_j^{2-p}b_{0,j}^p < \infty.
\end{equation}
Define the sequence of weights
\begin{equation}
\label{eq:omegas}
\omega_\nu := 2^{\|\nu\|_0}\vbf^\nu = 2^{\|\nu\|_0/2} \prod_{j: \nu_j \neq 0}v_j^{\nu_j}, \quad \text{for all } \nu \in \Fcal.
\end{equation}

Let $\varepsilon$ be a target accuracy and let $s$ be such that 
\begin{equation*}
2^{1/p-1}\sqrt{5}s^{1/2-1/p}\|u\|_{\Xcal,p} \leq \varepsilon.
\end{equation*}
Let $\Lambda := \{ \nu \in \Fcal: \omega_\nu^2 \leq s/2 \}$ such that $N := |\Lambda| < \infty$.
Let 
\begin{equation*}
m \geq c_0 s\log^3(s)\log(N)
\end{equation*}
and draw $m$ sampling points $\ybs^{(1)}, \cdots, \ybs^{(m)}$ i.i.d. from $\eta$ and let $b^{i} \in \Xcal$ be an approximation of $u(\ybs^{(i)})$ such that $\|b^i - u(\ybs^{(i)})\|_\Xcal \leq \varepsilon$. 
Finally, let $\wu$ be the solution to
\begin{equation}
\label{eq:minProbEDP}
\min_{z \in \Xcal \otimes P_\Lambda} \|z\|_{\Xcal,1}^{(\omega)}, \quad \text{s.t. } \|Az-b\|_{\Xcal,2} \leq 2\sqrt{m}\varepsilon.
\end{equation}
Then the following bounds hold
\begin{align*}
\|u-\wu\|_{L^\infty(\Xcal;\Ucal,\eta)} &\leq C s^{1-1/p}\|u\|_{\Xcal,p}^{(\omega)} + C'\sqrt{s}\varepsilon \\
\|u-\wu\|_{L^2(\Xcal;\Ucal,\eta)} &\leq D s^{1/2-1/p}\|u\|_{\Xcal,p}^{(\omega)} + D' \varepsilon.
\end{align*}

\end{atheorem}

The goal of the following part of this paper is to verify all the assumptions and conditions mentioned in the previous sections to prove the main theorem, by applying Theorem~\ref{thm:fctApprox} to this particular problem.

\subsection{Truncation to finite dimensional problems}
So far the problem is a continuous, infinite one and we need some ways to discretize and truncate it before handling its resolution numerically. 
This takes two forms: truncating the continuous function space of solution (typically what Finite Elements and similar discretization methods deal with) and truncating the space of polynomial to a finite one. 

As described in the previous section, the basic idea is to decouple the space and parameter variables. 
Hence given a (countable) orthonormal set $(\varphi_j)_{j \geq 1}$ of $\Xcal$ and a (countable) set of tensorized (\Tscheb) polynomials $(T_\nu)_{\nu \in \Fcal}$. 
The solution may then be expressed as 
$$
u(x;\ybs) = \sum_{\nu \in \Fcal} \sum_{j \geq 1}u_\nu^j \varphi_{j}(x) T_\nu(\ybs).
$$
The challenges will be to find finite dimensional spaces $\Xcal^h = \Xcal^\Jcal \subset \Xcal$ and $\Lambda \subset \Fcal$ such that the approximation
$$
\wu(x;\ybs) = \sum_{\nu \in \Lambda} \sum_{j \in \Jcal} \wu_\nu^j \varphi_j^h(x) T_\nu(\ybs)
$$
is close enough to the original sought after function. 

\subsubsection{Petrov-Galerkin discretization}

Throughout the rest of this note we assume given \emph{scale of smoothness spaces} $\{\Xcal_t\}_{0 \leq t \leq \bar{t}}$ such that 
\begin{equation}
\label{eq:scaleofsmmoothness}
\Xcal := \Xcal_0 \varsupsetneq \Xcal_1 \varsupsetneq \cdots \varsupsetneq \Xcal_{\bar{t}},
\end{equation}
where the spaces are defined by interpolation for the non-integer indices. 

We also consider at our disposal a one-parameter family of finite-dimensional spaces $\{\Xcal^h\}_{h > 0}$ with $N_h := \operatorname{dim}(\Xcal^h) < \infty$.
We assume that the spaces $\{\Xcal^h\}_h$ are dense in $\Xcal$ as $h \to 0$.

\begin{assumption}[Approximation property of the discrete spaces]
\label{assump:AP}
We assume the that discretization spaces have the \emph{approximation property} in the smoothness scale~\eqref{eq:scaleofsmmoothness}: for $0 < t \leq \bar{t}$, there exists a constant $C_t$, such that for all $0 < h \leq 1$, and all $u \in \Xcal_t$, it holds
\begin{equation}
\label{eq:APapproximation}
\inf_{u^h \in \Xcal^h}\|u-u^h\|_\Xcal \leq C_t h^t\|u\|_{\Xcal_t}.
\end{equation}
\end{assumption}

Namely, we say that the family of operators $A(\ybs)$ is $\inf-\sup$-stable in the discretization spaces if there exist constants $\mu_d> 0$ and $h_0 > 0$ such that for all $0 < h \leq h_0$
\begin{equation}
\label{eq:infsupDiscrete}
\begin{array}{c}
\inf_{0 \neq v \in \Xcal^h} \sup_{0 \neq w \in \Ycal^h} \frac{\left\langle A(\ybs)v,w \right\rangle}{\|v\|_\Xcal\|w\|_\Ycal} \geq \mu_d, \\
\inf_{0 \neq w \in \Ycal^h} \sup_{0 \neq v \in \Xcal^h} \frac{\left\langle A(\ybs)v,w \right\rangle}{\|v\|_\Xcal\|w\|_\Ycal} \geq \mu_d.
\end{array}
\end{equation}

We recall the following classical result (see for example~\cite[Chapter 6]{Brezzi13book}).
\begin{aprop}
\label{prop:linGh}
Let $\Xcal^h$ and $\Ycal^h$ be discretization spaces for the PG method, 
such that the uniform discrete $\inf-\sup$ conditions~\eqref{eq:infsupDiscrete}
are fulfilled and assume that the bilinear operator 
$\Xcal \times \Ycal \ni (u,w) \mapsto \langle A(\ybs)u,w\rangle$ is continuous,
uniformly with respect to $\ybs\in U$.

Then the PG projections $G^h(\ybs): \Xcal \to \Xcal^h$ 
are well-defined linear operators, 
whose norms are uniformly bounded with respect to the 
parameters $\ybs$ and $h$, i.e.,
\begin{align}
\label{eq:priorBound}
\sup_{\ybs\in U} \sup_{h>0}
\|u^h(\ybs)\|_{\Xcal} &\leq \frac{1}{\mu_d}\|f\|_{\Ycal'}, 
\\
\label{eq:normProjection}
\sup_{\ybs \in U} \sup_{h > 0} \|G^h(\ybs)\|_{\Lcal(\Xcal)} &\leq \frac{C}{\mu_d}
\end{align}
The Galerkin projections are uniformly quasi-optimal:
for every $\ybs\in \Ucal$ we have the a-priori error bound
\begin{equation}
\label{eq:quasiOptimal}
\|u(\ybs) - u^h(\ybs)\|_\Xcal
\leq 
\left(1+\frac{C}{\mu_d}\right)\operatorname{inf}_{v^h \in \Xcal^h} \|u(\ybs) - v^h\|_\Xcal
\;.
\end{equation}
\end{aprop}

As a consequence, combining Eq.~\eqref{eq:quasiOptimal} with Eq.~\eqref{eq:APapproximation} ensures us the following approximation for every $u \in \Xcal_t$: 
\begin{equation}
\label{eq:spatialApprox}
\|u - G_hu\|_\Xcal \leq C'_{t} h^t \|u\|_{\Xcal_t}. 
\end{equation}
An important point to notice is that~\eqref{eq:spatialApprox} is valid uniformly for all vectors $\ybs \in \Ucal$.
The goal of the compressed sensing approach is to compute the finite-dimensional approximation $G_hu$, hereby breaking (one part of) the infinite dimensionality of the problem.

\subsubsection{Parameter truncation}
Following~\eqref{eq:spatialApprox}, we can discretize the infinite dimensional minimization to a manageable, finite dimensional one.
Looking at the parameter space, we still face two major problems: the first one is that the number of parameters may be infinite (but countable) and the second is that the set of multi-indices (see~\eqref{eq:fcal}) that we use is also infinite (countable) even if the number of parameters is finite. 

Wlog we can assume that the $A_j$ are ordered in decreasing order of their energies, i.e. $b_{0,j} \geq b_{0,j+1}$, for all $j$. 
We assume some decay of the \emph{energy} of the operator $A(\ybs)$ 
such that for any $\varepsilon > 0$, 
there exists $\tau := \tau(\varepsilon,A)$ with 
\begin{equation}
\label{eq:decayEnergyOperator}
\|A(\ybs) - A^{(\tau)}(\ybs)\|_{\Lcal(\Xcal, \Ycal')} \leq \varepsilon \mu, \quad \forall \ybs \in \Ucal,
\end{equation}
where $\mu$ is the constant appearing in the $\inf-\sup$ conditions and 
we define the weak solutions of the truncated version of Eq.~\eqref{eq:affineFamily}:
\begin{equation}
\label{eq:weakTruncated}
\text{Find $u^{(\tau)} \in \Xcal$, such that } \langle A^{(\tau)}(\ybs)u^{(\tau)}, v \rangle = \langle f, v \rangle \quad \mbox{ for all } v \in \Ycal,
\end{equation}
with the operator $A^{(\tau)}(\ybs)$ defined for a finite $\tau \in \Nbb$ as $A(y_1, y_2, \cdots, y_\tau, 0, 0, \cdots)$.

In this case, the following result, taken from \cite{Bouchot16MLCSPG}, generalizing results in~\cite{Dick13QMCPG} holds.
\begin{aprop}
\label{prop:truncation}
Assume the family of operators $A$ satisfy the $\inf-\sup$ conditions
and the decay property~\eqref{eq:decayEnergyOperator}. 
Then for any accuracy parameter $\varepsilon$, 
there exists a truncation parameter $\tau \in \Nbb$ such that the solutions 
to the truncated problem~\eqref{eq:weakTruncated} and 
to the original problem~\eqref{eq:affineFamily} are close to each other in the following sense 
\begin{align}
\label{eq:approxDecayEnergyOperator}
\|u^{(\tau)}(\ybs) - u(\ybs)\|_{\Xcal} &\leq \frac{C \varepsilon}{\mu}\|f\|_{\Ycal'}, 
\end{align}
where $u^{(\tau)}(\ybs)$ is the solution of the truncated problem~\eqref{eq:weakTruncated}.
\end{aprop}
A sufficient condition for~\eqref{eq:decayEnergyOperator} to hold is to have $(b_{0,j})_{j \geq 1} \in \ell^1(\Nbb)$.
A consequence of this result is that it is possible to sample only in the finite dimensional spaces $[-1,1]^\tau$ instead of in the infinite space $\Ucal$. 
One last step towards lowering everything to a finite dimensional problem is to truncate the polynomial space used for the approximation.

To this end we need an infinite dimensional version of Theorem~\ref{thm:fctApprox}.
This theorem being of interest on its own, we have written it in a very general form, independently from the parametric PDE application that we have in mind. 
\begin{atheorem}
\label{thm:approxInfDim}
Suppose $(T_\nu)_{\nu \in \Fcal}$ is a countable orthonormal system, indexed by $\Fcal$, with orthogonalization measure $\eta$. 
Assume given some weights $(\omega_\nu)_{\nu \in \Fcal}$ such that $\omega_\nu \geq \|T_\nu\|_\infty$ for all $\nu \in \Fcal$. 
For a parameter $s \geq 1$ define $\Lambda := \{\nu \in \Fcal: \omega_\nu^2 \leq s/2\}$ and assume $N := |\Lambda| < \infty$
and let 
\begin{equation}
\label{eq:m}
m \geq c_0 s\log^3(s)\log(N).
\end{equation}
For a function $f \in \Xcal \otimes P_\Gamma$ with $f(x;\ybs) = \sum_{\nu \in \Fcal} f_\nu(x)T_\nu(\ybs)$ and $f_\nu \in \Xcal$, 
draw $m$ samples $\ybs^{(1)}, \cdots, \ybs^{(m)}$ i.i.d. from $\eta$ and let $A := (T_\nu(\ybs^{(i)}))_{1 \leq i \leq m; \nu \in \Lambda}$ be the sampling matrix. 
Assume that there exists a $0 < p < 1$ such that $\|f\|_{\Xcal,p}^{(\omega)} < \infty$.
Let $\varepsilon > 0$ such that 
\begin{equation}
\label{eq:eps4Approx}
\varepsilon \geq 2^{1/p-1}\sqrt{5}s^{1/2-1/p}\|f-f_{\Lambda}\|_{\Xcal,p}^{(\omega)}.
\end{equation} 
Let $b^{i}$ be an approximation of $f(\ybs^{(i)})$ such that $\|b^i - f(\ybs^{(i)})\|_\Xcal \leq \varepsilon$ and 
let $\wz$ be the solution of 
\begin{equation*}
\min\|z\|_{\Xcal,1}^{(\omega)}, \quad \text{s.t. } \|Az-b\|_{\Xcal,2} \leq 2\sqrt{m}\varepsilon.
\end{equation*}
The following bounds hold for the function $\wf(x;\ybs) := \sum_{\nu \in \Lambda} T_\nu(\ybs) \wz_{\nu}(x)$
\begin{align*}
\|f-\wf\|_{L^\infty} \leq \|f-\wf\|_{2,1}^{(\omega)} &\leq  c\sigma_{s/2}(f)_{\Xcal,1}^{(\omega)} + d \sqrt{s}\varepsilon,\\
\|f-\wf\|_{L^2} = &\leq \frac{c'}{\sqrt{s}}\sigma_{s/2}(f)_{2,1}^{(\omega)} + d'\varepsilon.
\end{align*}
with $c,d$ and $c',d'$ universal constants. 
 
\end{atheorem}

\begin{proof}
This results is obtained by modifying the proof of~\cite[Theorem 1.2]{Rauhut13wCS} to the joint sparsity structure.
The whole idea is to split the infinite extension into an \emph{interesting} part and a \emph{residual}; \jlb{this corresponds to a decomposition of the function $f = \sum_{\nu \in \Fcal} f_\nu = \sum_{\nu \in \Lambda}f_\nu + \sum_{\nu \in R}f_\nu =: f_\Lambda + f_R$, where we define the residual set as $R = \Fcal \backslash \Lambda$.}
Our goal is to show that the residual is small enough, given the number of samples. 

Let $\ybs^{(1)}, \cdots, \ybs^{(m)}$ be $m$ samples drawn i.i.d. from the orthogonalization measure $\eta$. 
Following the calculations leading to~\eqref{eq:approxBochner2}, we obtain 
$$
\Ebb\left[ \|f_R(\ybs^{(i)})\|_{\Xcal}^2 \right] = \int \limits_{\ybs \in \Ucal} \|{f_R}\|_{\Xcal}^2\dif{\eta}(\ybs) = \|f_R\|_{\Xcal,2}^2 = \|F_R\|_{2,2}^2,
$$
where $F_R \jlb{:= (f_\nu^j)_{\nu \in R; j \in \Jcal}}$ denotes the infinite matrix containing the coefficients $f_\nu$ \jlb{($\nu \in R$)} expanded on the basis of $\Xcal$.

Furthermore, by definition of the sets $\Lambda$ and $R$, we have $\omega_\nu^2 > s/2$ for all $\nu \in R$ hence 
$$
\|f_R\|_{\Xcal,2}^2 \leq \frac{2}{s}\sum_{\nu \in R}\omega_\nu^2\|f_\nu\|_{\Xcal}^2 \leq \frac{2}{s} \left(\sum_{\nu \in R} \omega_\nu \|f_\nu\|_{\Xcal} \right)^2 = \frac{2}{s}{\|f_R\|_{2,1}^{(\omega)}}^2.
$$

Therefore the random variable $z_\ell := \|f_R(\ybs^{(\ell)})\|_{\Xcal}^2 - \Ebb\left[\|f_R(\ybs^{(\ell)})\|_{\Xcal}^2\right]$ is a zero mean random variable with variance 
$$
\Ebb\left[ z_\ell^2 \right] = \Ebb \left[ \left(\|f_R(\ybs^{(\ell)})\|_{\Xcal}^2 - \Ebb\left[\|f_R(\ybs^{(\ell)})\|_{\Xcal}^2\right]\right)^2 \right] \leq \Ebb\left[ \|f_R(\ybs^{(\ell)})\|_{\Xcal}^4 \right].
$$
Moreover, for any $\ybs \in \Ucal$, we have $\|f_R(\ybs)\|_{\Xcal} = \|\sum_{\nu \in R} f_\nu T_\nu(\ybs)\|_{\Xcal} \leq \sum_{\nu \in R}\|T_\nu\|_\infty\|f_\nu\|_{\Xcal} \leq \|f_R\|_{\Xcal,1}^{(\omega)}$, we obtain 
$$
\Ebb\left[ z_\ell^2 \right]  \leq {\|f_R\|_{\Xcal,1}^{(\omega)}}^2\Ebb\left[ \|f_R(\ybs^{(\ell)})\|_{\Xcal}^2 \right] \leq \frac{2}{s}{\|f_R\|_{\Xcal,1}^{(\omega)}}^4.
$$
Note that this last quantity is bounded by our assumption. 

Finally, applying Bernstein's inequality, for all $t > 0$ 
$$
\Pbb\left( \left| \frac{1}{m}\sum_{\ell=1}^m  \left\|f_R\left(\ybs^{(\ell)}\right)\right\|_{\Xcal}^2 - \|f_R\|_{\Xcal,2}^2 \right| \geq t \right) \leq \exp\left( -\frac{mt^2/2}{2{\|f_R\|_{\Xcal,1}^{\omega}}^4/s + t {\|f_R\|_{\Xcal,1}^{(\omega)}}^2/3} \right).
$$
Setting $t = \frac{3}{s}{\|f_R\|_{\Xcal,1}^{(\omega)}}^2$ gives 
$$
\Pbb\left( \left| \frac{1}{m}\sum_{\ell=1}^m  \left\|f_R\left(\ybs^{(\ell)}\right)\right\|_{\Xcal}^2 - \|f_R\|_{\Xcal,2}^2 \right| \geq \frac{3}{s}{\|f_R\|_{\Xcal,1}^{(\omega)}}^2 \right) \leq \exp\left( - \frac{3m}{2s} \right).
$$
Plugging in the number of samples $m$ in Eq~~\eqref{eq:m} gives the probability in the truncation error 
$$
\Pbb\left( \sum_{\ell^=1}^m \left\|f_R\left(\ybs^{(\ell)}\right)\right\|_{\Xcal}^2 \geq \frac{5m}{s}{\|f_R\|_{\Xcal,1}^{(\omega)}}^2 \right) \leq N^{-\log^3(s)}
$$
By applying Stechkin's bound~\eqref{eq:Stechkin} and the definition $\Lambda$, we may bound
$$
\|f_R\|_{\Xcal,1}^{(\omega)} \leq \sigma_{s/2}(f)_{\Xcal,1}^{(\omega)} \left( \frac{1}{s/2} \right)^{1/p-1}\|f\|_{\Xcal,p}^{(\omega)} = 2^{1-1/p}2^{1/p-1}\|f\|_{\Xcal,p}^{(\omega)}.
$$
It follows from Eq.~\eqref{eq:eps4Approx}, that 
$$
\Pbb\left( \sum_{\ell^=1}^m \left\|f_R\left(\ybs^{(\ell)}\right)\right\|_{\Xcal}^2 \geq m\varepsilon^2 \right) \leq \Pbb\left( \sum_{\ell^=1}^m \left\|f_R\left(\ybs^{(\ell)}\right)\right\|_{\Xcal}^2 \geq \frac{5m}{s}{\|f_R\|_{\Xcal,1}^{(\omega)}}^2 \right) \leq N^{-\log^3(s)}.
$$

By assumption, we have $\|b^{(i)} f(\ybs^{(i)})\|_\Xcal \leq \varepsilon$ for all $1 \leq i \leq m$. 
Hence, altogether
$\|b - f_\Lambda\|_{\Xcal,2} \leq 2\sqrt{m}\varepsilon.$

We may now apply Theorem~\ref{thm:fctApproxInf} which gives the bounds on the error. 

\end{proof}

\subsection{Proof of Theorem~\ref{thm:EDPrecovery}}
We are now ready to prove our Theorem. 

\subsubsection{Varifying the assumptions}

The existence of solutions is given by the summability of the weight sequence $\vbf$ in Eq.~\eqref{eq:wuea}. 
And it follows from Proposition~\ref{prop:linGh} that we may approximate our solutions by Petrov-Galerkin discretization. 

Assume we are given samples $b^{(l)}$ which corresponds to the Petrov-Galerkin discretization of the truncated operator. 
Because of Proposition~\ref{prop:truncation}, Proposition~\ref{prop:linGh} and the following Eq.~\eqref{eq:spatialApprox}, we can find solvers such that
\begin{equation*}
\|b^{(l)} - u(\ybs^{(l)})\|_{\Xcal} \leq \varepsilon. 
\end{equation*}

The weighted $\ell^p$ summability of the solution $u$ in $\Fcal$ requires a result from~\cite{Rauhut14CSPG} 
\begin{atheorem}
\label{thm:summability}
Let $0 < p \leq 1$. Assume that the summability conditions ~\eqref{eq:lpSummability} and~\eqref{eq:wuea} hold for some weight sequence $\vbf$. 
For $\theta \geq 1$ construct a sequence of weights $\omega = (\omega_\nu)_{\nu \in \Fcal}$ on $\Fcal$ with $\omega_\nu = \theta^{\|\nu\|_0}\vbf^\nu$. 
Then the sequence of norms $(\|u_\nu\|_{\Xcal^h})_{\nu \in \Fcal} \in \ell^p_{\omega}(\Fcal)$. 
\end{atheorem}
In particular, for $\theta = \sqrt{2}$, we ensure that 1) $\omega_\nu \geq \|T_\nu\|_\infty$ (see~\eqref{eq:linfTscheb}) and 2) $\|u\|_{\Xcal,p} = \left(\sum\|u_\nu\|_\Xcal^p\right)^{1/p} < \infty$.
It then follows from the weighted Stechkin bound from Proposition~\ref{prop:Stechkin} and Eq.~\eqref{eq:Stechkin} that
\begin{equation*}
\sigma_s(u)_{r,p}^{(\omega)} \leq \widetilde{\sigma}_s(u)_{r,p}^{(\omega)} \leq (s - \|\omega\|_\infty^2)^{\frac{1}{p} - \frac{1}{q}}\|u\|_{r,q}^{(\omega)}
\end{equation*}
so that when plugged back into the a priori bounds obtained from Theorem~\ref{thm:approxInfDim} we arrive at 
\begin{align*}
\|u-\wu\|_{L^\infty(\Xcal;\Ucal,\eta)} &\leq c\sigma_{s/2}(u)_{\Xcal,1}^{(\omega)} + c' \sqrt{s} \varepsilon + \|f_R\|_{L^\infty(\Xcal;\Ucal,\eta)} \\
	&\leq c\sigma_{s/2}(u)_{\Xcal,1}^{(\omega)} + c' \sqrt{s} \varepsilon + \|f_R\|_{\Xcal,1}^{(\omega)} \leq (c+1) \sigma_{s/2}(u)_{\Xcal,1}^{(\omega)} + c'\sqrt{s}\varepsilon,
\end{align*}
where the first inequality is a computation done in the proof of Theorem~\ref{thm:approxInfDim}.
The bound is then obtained by applying Stechkin's bound. 
We can make the same computations for the $L^2$ error bound, using the fact that $\|u_R\|_{L^2(\Xcal;\Ucal,\eta)} \leq \sqrt{\frac{2}{s}}\|u_R\|_{\Xcal,1}^{(\omega)}$.

\subsubsection{A remark on the size of the parameter space}
It is important to notice that the theorem is valid for a space of multi-indices $\Lambda$ with $N$ elements. 
The theorem is only relevant in case the size of this space remains \emph{small}. 
To this end, we recall the following results from~\cite{Rauhut14CSPG}: 
\begin{aprop}
Assume $s \geq 1$. 
\begin{enumerate}
\item (constant weights) Let $\omega_j = \beta, j = 1 \dots d$, for some $\beta >1$ and $\omega_j = \infty$ for $j > d$. Then 
$$
N \leq \left\{ \begin{array}{l} 
(\log_{\beta^2}(\beta^2s/2) )^d, \qquad \text{ if } d \leq \log_{2\beta^2}(s/2), \\ 
\left(  \left( 1+\frac{1}{ \log_2(\beta^2) } \right) ed \right)^{\log_{2\beta^2}(s/2)}, \qquad \text{otherwise. }
\end{array} \right.
$$
\item (polynomially growing weights) Let $\omega_j = cj^\alpha$ for $j \geq 1$ and some $c > 1$ and $\alpha > 0$. 
Then there exist constants $C_{\alpha,c} > 0$ and $\gamma_{\alpha,c}$ such that 
$$
N \leq C_{\alpha,c} s^{\gamma_{\alpha,c} \log(s)}.
$$
\end{enumerate}
\end{aprop}

Plugging back these estimates into the main theorem means that the total number of sampling is linear, with polylog factors, in the weighted sparsity $s$.

\section{Final remarks}
A weakness which has not been addressed here is the reliance of the optimization problem on the knowledge of the weighted $\ell^p$ norm of the expansion. 
Some recent results have shown that such optimization problem are stable when the noise is not known~\cite{Paglia18UnknownNoise}. 

Another approach to overcome the necessity of knowing the noise is to adapt greedy and iterative approximation algorithms (such as Orthogonal Matching Pursuit, Hard Thresholding Pursuit~\cite{Foucart11HTP} and its Graded variants~\cite{Bouchot13GHTP,Bouchot14fHTP}, Null-space tuning~\cite{Li14NST}, CoSaMP~\cite{needell2009cosamp} or Iterative Hard Thresholding~\cite{blda09,Jo2013wIHT}). 
These algorithms need to be analyzed in the context of weighted sparsity and then adapted to fit the vector-valued function recovery. 
This has only recently started being tackled in research~\cite{adcock2018wOMP} and already shows quite a lot of potential.




\bibliographystyle{abbrv}
\bibliography{MLCSPG}

\appendix
\section{Quasi-approximations}
\label{app:quasiapprox}

\subsection{Proof of inequality}\eqref{eq:estimateQuasiweighted}
Without loss of generality, assume that the entries of $\xbf$ are (weighted) ordered, i.e. $\xbf = \widetilde{\xbf}$. 
Let $k_{3s} := \max\{k: \sum_{i=1}^k \omega_i^2 \leq 3s\}$, and $T$ be the subset $\{1,2, \cdots, k_{3s}\}$.
By definition, it holds $\omega(T) \leq 3s$. 
By virtue of $\|\omega\|_\infty^2 \leq s$ it also holds that $\omega(T) \geq 2s$ (otherwise, we could add another index to the set $T$ contradicting the optimality of $k_{3s}$).
Similarly, let $S$ be the support of the quasi-best weighted $s$ block approximation, i.e. $S = \{1, \cdots, k_s\}$.
Note that since $\omega_j \geq 1$ it follows directly, that the true (=non-weighted) cardinality of $S$ fulfills $|S| \leq s$.
Define $n_b = \lfloor \omega_b^2 +1 \rfloor$ and consider the residue $r_b = n_b - \omega_b^2$ for $b \in \Bcal$. 
We have 
$$ 
\sum_{j \in S} n_j \leq \sum_{j \in S} \omega_j^2 + s \leq 2s \leq \omega(T) \leq \sum_{j \in T} n_j.
$$ 	

The goal of the proof is to group copies of blocks which themselves embed the weight information. 
Let then, for a given inner norm $\|\cdot\|$ acting on the blocks (not necessarily the same for each block)
$\zbf = \left( \|\xbf[b]\|^p/\omega_b^p {\bf{1}}_{n_b-1}, (1-r_b)\|\xbf[b]\|/\omega_b^p \right)_{b \in \Bcal}$
where ${\bf{1}}_{n_b-1}$ denotes the row vector containing $n_b-1$ ones. 
Altogether, the vector $\zbf$ contains (ordered) $n_b$ terms which sum up to $\|\xbf[b]\|^p\omega_b^{2-p}$ for all $b \in \Bcal$. 
Finally
\begin{align*}
{\|\xbf\|_{\|\cdot\|/p}^{(\omega)}}^p &= \max\left\{ \sum_{b \in S}\omega_b^2\|\xbf[b]\|^p/\omega_b^p: S \subset \Bcal, \omega(S) \leq s \right\} \\
&\leq \max\left\{ \sum_{b \in S}\omega_b^2\|\xbf[b]\|^p/\omega_b^p: S \subset \Bcal, \sum_{b \in S}n_b \leq \sum_{b \in T}n_b \right\} \\
&\leq \max\left\{ \sum_{b \in S} \|\zbf[b]\|_1: S \subset \Bcal, |S| \leq \sum_{b \in T}n_b \right\} \leq {\|\xbf_T\|_{\|\cdot\|/p}^{(\omega)}}^p.
\end{align*}
In the last line, we have used a block structure associated with the vector $\zbf$ which is inherited from $\xbf$. 

\subsection{Proof of the block weighted Stechkin bound}
\begin{aprop}
Let $\Bcal = (\Bcal_1, \cdots, \Bcal_B)$ be a structure for $\Rbb^N$. 
Let $q < p \leq 2$. 
For any $\xbf \in \Rbb^N$, let $\xbf[S]$ define its quasi-best weigthed $s$ term approximation and define $\tilde{\sigma}_s{\xbf}_p^{(\omega)} := \|\xbf-\xbf[S]\|_{2,p}^{(\omega)}$. 
The following Stechkin's bound hold
\begin{equation}
\sigma_s(\xbf)_p^{(\omega)} \leq \tilde{\sigma}(\xbf)_{p}^{(\omega)} \leq \left( s - \|\omega\|_\infty^2 \right)^{1/p-1/q} \|\xbf\|_{2/q}^{(\omega)}.
\end{equation}
\end{aprop}
\begin{proof}
With $\xbf[S]$ the quasi-best weigted $s$ term approximation and $p < q \leq 2$ it follows 
\begin{align}
\left( \tilde{\sigma}_s(\xbf)_p^{(\omega)} \right)^p &= \sum_{j \notin S}\|\xbf[\Bcal_j]\|_2^p\omega_j^{2-p} \leq \max_{j \notin S}\left\{ \|\xbf[\Bcal_j]\|_2^{p-q}\omega_j^{q-p} \right\}\sum_{j \notin S}\|\xbf[\Bcal_j]\|_2^{q}\omega_j^{2-q} \nonumber \\ 
\label{eq:stechkinStar} &\leq \left( \max_{j \notin S}\|\xbf[\Bcal_j]\|_2 \omega_j^{-1} \right)^{p-q} {\|\xbf\|_{2/q}^{(\omega)}}^q
\end{align}
By definition of the set $S$ of the quasi-best approximation of $\xbf$, we have that $s- \|\omega\|_\infty^2 \leq s - \omega_j^2\leq \omega(S)$
Introducing $\lambda_k := \left(\sum_{j \in S}\omega_j^2\right)^{-1}\omega_k \leq (s- \|\omega\|_\infty^2)^{-1}\omega_k^2$, we arrive at
\begin{align*}
\left( \max_{j \notin S} \|\xbf[\Bcal_j]\|_2 \omega_j^{-1} \right)^q &\leq \sum_{k \in S}\lambda_k\|\xbf[\Bcal_k]\|_2^q \omega_k^{-q} \leq \left( s - \|\omega\|_\infty^2 \right)^{-1}\sum_{k \in S}\omega^{2-q}\|\xbf[\Bcal_k]\|_2^q \leq \left( s - \|\omega\|_\infty^2 \right)^{-1} {\|\xbf\|_{2/q}^{(\omega)}}^q.
\end{align*} 
Injecting this into~\eqref{eq:stechkinStar}, we finally arrive at 
\begin{equation}
\left( \tilde{\sigma}_s(\xbf)_p^{(\omega)} \right)^p \leq \left(  \left( s - \|\omega\|_\infty^2\right)^{-1}\right)^{\frac{p-q}{q}}{\|\xbf\|_{2/q}^{(\omega)}}^{p-q}{\|\xbf\|_{2/q}^{(\omega)}}^q = \left(s - \|\omega\|_\infty^2\right)^{1-p/q}{\|\xbf\|_{2/q}^{(\omega)}}^p.
\end{equation}
\end{proof}

\subsection{Proof of the weighted block RIP recovery results}
\label{a:proofOfLemma}

We use the following results, known in traditional compressed sensing~\cite[Proposition 6.3]{Foucart13book} and in group sparsity~\cite[Eq. 40]{Eldar2009group}.

\begin{alemma}
\label{lemma:wbripST}
Let $\Bcal = (\Bcal_1, \cdots, \Bcal_B)$ be a block structure and $\omega = (\omega_1, \cdots, \omega_B)$ with $\omega_i \geq 1$ its associated weight sequence. Let $\ubf, \vbf \in \Rbb^N$ be two block vectors such that $\|\ubf\|_0^{(\omega)} \leq s$ and $\|\vbf\|_0^{(\omega)} \leq t$. 
In addition, assume that $\bsupp(\ubf) \cap \bsupp(\vbf) = \varnothing$. 
Then 
\begin{equation}
|\langle A\ubf, A\vbf\rangle| \leq \delta_{s+t}\|\ubf\|_2 \|\vbf\|_2.
\end{equation}
\end{alemma}

\begin{proof}
Let $S := \operatorname{B-supp}(\ubf) \cup \operatorname{B-supp}(\vbf)$. It follows that $\omega(S) \leq s+t$. 
\begin{align}
|\langle A\ubf, A\vbf\rangle| = |\langle \left(A_S^*A_S-I\right)\ubf, \vbf \rangle| \leq \|A_S^*A_S-I\|_{2 \to 2}\|\ubf\|_{2,2}\|\vbf\|_{2,2},
\end{align}
where we defined as usual as 
\begin{equation}
\|A\|_{2 \to 2} := \sup_{\xbf \neq 0} \frac{\|A\xbf\|_2}{\|\xbf\|_{2,2}}.
\end{equation}
Now note that the RIP like bound is obtained as 
\begin{align*}
\|A\xbf\|_2^2 - \|\xbf\|_{2,2}^2 &= \langle A\ubf, A\ubf\rangle - \langle\xbf,\xbf\rangle \leq \|A_S^*A_S - I\|_{2 \to 2} \|\xbf\|_{2,2}^2. 
\end{align*}
Hence, by definition of the RIP constant, we have $\|A_S^*A_S - I\|_{2 \to 2} \leq \delta_{s+t}$.
\end{proof}

\end{document}